\documentclass[11pt,a4paper,oneside]{amsart}
\makeatletter
\@namedef{subjclassname@2010}{%
  \textup{2010} Mathematics Subject Classification}
\makeatother
\usepackage{a4wide}
\usepackage{enumerate}

\usepackage{amsfonts}
\usepackage{amssymb}
\usepackage{amsmath}
\usepackage{amsthm}

\usepackage[all]{xy}
\usepackage{graphicx}
\usepackage{mathrsfs}
\usepackage{color}
\usepackage{url}
\usepackage{array}
\usepackage{pifont}
\usepackage{caption}
\usepackage{booktabs}
\usepackage{setspace}

\allowdisplaybreaks

\newtheorem{thm}{Theorem}[section]
\newtheorem{lem}[thm]{Lemma}

\newtheorem{note}[thm]{Note}
\newtheorem{mukaispecseq}[thm]{Mukai Spectral Sequence}
\newtheorem{dualspecseq}[thm]{``Duality'' Spectral Sequence}
\newtheorem{specseq}[thm]{Spectral Sequence}

\newtheorem{prop}[thm]{Proposition}
\newtheorem{conj}[thm]{Conjecture}

\newtheorem{defn}[thm]{Definition}
\newtheorem{notation}[thm]{Notation}
\theoremstyle{remark}
\newtheorem{rem}[thm]{Remark}
\newtheorem{exmp}[thm]{Example}
\newtheorem*{theorem*}{Theorem}
\newtheorem*{notation*}{Notation}
\theoremstyle{remark}

\numberwithin{equation}{section}

\newcommand{\m}{\mathbf{m}}

\newcommand{\HX}{\widehat{X}}

\newcommand{\coh}{\operatorname{Coh}}
\newcommand{\adiag}{\operatorname{adiag}}
\newcommand{\Supp}{\operatorname{Supp}}

\newcommand{\ch}{\operatorname{ch}}

\newcommand{\Aut}{\operatorname{Aut}}
\newcommand{\GL}{\operatorname{GL}}
\newcommand{\SL}{\operatorname{SL}}

\newcommand{\Ext}{\operatorname{Ext}}
\newcommand{\calEnd}{\operatorname{\mathcal{E}\textit{nd}}}
\newcommand{\calHom}{\operatorname{\mathcal{H}\textit{om}}}
\newcommand{\calExt}{\operatorname{\mathcal{E}\textit{xt}}}
\newcommand{\Hom}{\operatorname{Hom}}

\newcommand{\id}{\operatorname{id}}
\newcommand{\rk}{\operatorname{rk}}
\newcommand{\RR}{\operatorname{\mathbf{R}}}

\newcommand{\Pic}{\operatorname{\mathbf{Pic}}}
\newcommand{\HN}{\operatorname{HN}}
\newcommand{\NS}{\operatorname{NS}}

\newcommand{\SH}{\operatorname{\scriptscriptstyle{H}}}

\newcommand{\Up}{\Upsilon}
\newcommand{\HUp}{\widehat{\Upsilon}}
\newcommand{\TUp}{\widetilde{\Upsilon}}
\newcommand{\HPi}{\widehat{\Pi}}
\newcommand{\TPi}{\widetilde{\Pi}}
\newcommand{\HGa}{\widehat{\Gamma}}
\newcommand{\Ga}{\Gamma}

\newcommand{\om}{\omega}

\newcommand{\A}{\mathcal{A}}
\newcommand{\B}{\mathcal{B}}
\newcommand{\D}{\mathcal{D}}
\newcommand{\E}{\mathcal{E}}
\newcommand{\F}{\mathcal{F}}

\newcommand{\PP}{\mathcal{P}}
\newcommand{\OO}{\mathcal{O}}

\newcommand{\T}{\mathcal{T}}
\newcommand{\U}{\mathcal{U}}

\newcommand{\SC}{\mathscr{C}}

\newcommand{\HH}{\mathscr{H}}
\newcommand{\SM}{\mathscr{M}}

\newcommand{\Po}{\mathscr{P}}

\newcommand{\C}{\mathbb{C}}

\newcommand{\Q}{\mathbb{Q}}
\newcommand{\R}{\mathbb{R}}

\newcommand{\Z}{\mathbb{Z}}

\begin{document}

\title[Bridgeland Stability Conditions on Abelian Threefolds]
{Fourier-Mukai Transforms and Bridgeland Stability Conditions on
  Abelian Threefolds II}
\author[Antony Maciocia \& Dulip Piyaratne]{Antony Maciocia \& Dulip Piyaratne}

\address{AM: School of Mathematics \\ The University of Edinburgh \\
The King's Buildings \\ Peter Guthrie Tait Road \\
Edinburgh EH9 3FD  \\    UK.}
\email{ A.Maciocia@ed.ac.uk }

\address{DP: Kavli Institute for the Physics and Mathematics of the Universe (WPI)\\ The University of Tokyo Institutes for Advanced Study \\ The University of Tokyo \\ Kashiwa \\ Chiba 277-8583 \\ Japan.}
\email{ dulip.piyaratne@ipmu.jp }


\date{\today}



\begin{abstract}
We show that the conjectural construction proposed by  Bayer, Bertram,
Macr\'i and Toda  gives rise to  Bridgeland stability conditions for a
principally polarized abelian threefold with Picard rank one by
proving that tilt stable objects satisfy the strong  Bogomolov-Gieseker
type inequality. This is done by showing certain Fourier-Mukai transforms
give equivalences of abelian categories which are double tilts of
coherent sheaves.
\end{abstract}

\maketitle
\section*{Introduction}
\label{section0}
There is a growing interest in the study of Bridgeland stability
conditions on varieties. This notion was introduced in
\cite{Bri1} and some known examples can be found in \cite{Bri2, ABL, Okada, Macri1, Macri2, MP, Sch}.
See  \cite{Huy3} and \cite[Appendix D]{BBR} for comprehensive expositions on the subject.
Construction of such stability conditions on a given Calabi-Yau
threefold is an important problem but the only known example is on
an abelian threefold (see \cite{MP}). For further motivation from, for
example, Mathematical Physics, see \cite{Asp, Tod1}.
A conjectural construction of such a stability condition for any
projective threefold was introduced
by  Bayer, Bertram, Macr\'i and Toda in \cite{BMT, BBMT}.
Here they introduced the notion of tilt stability for objects in an
abelian subcategory  of the derived category which is a tilt of
coherent sheaves. These have now been studied extensively: \cite{Tod2, LM, BMT, Macri2, MP, Sch}.
This conjectural construction has boiled down to the requirement that
certain tilt stable objects satisfy a so-called (weak) Bogomolov-Gieseker (B-G for
short) type inequality. Moreover they went in to propose a stronger
version of this inequality which is known to hold for
projective 3-space (see \cite{Macri2}) and smooth quadric threefold (see \cite{Sch}).
In \cite{MP}, for a principally
polarized abelian threefold, we prove that tilt stable
objects satisfy the weak B-G type inequality associated to a special
complexified ample class. It was achieved by establishing an equivalence of
two abelian categories given by the classical Fourier-Mukai transform
with kernel the  Poincar\'e bundle. The aim of this paper is to
extend those ideas for certain kind of non-trivial Fourier-Mukai transforms (FMT
for short) to establish the {\em strong} B-G
type inequality for the same abelian threefold.

For an abelian variety $X$, the group of FMTs
$\Aut D^b(X)$ is well understood  via the notion of isometric
isomorphism (see \cite{Orl2} or \cite[Chapter 9]{Huy1}). To any FMT $\Phi_\E$ with kernel
$\E$  Orlov constructed an isometric automorphism $f_\E$ of the
product $X \times \HX$. He showed that $\Phi_\E \mapsto
f_\E$ is a surjective map of
groups and its kernel consists of trivial FMTs (which send skyscraper
sheaves to skyscraper sheaves up to shift) which also preserve
$\Pic^0(X)$ up to shift. 
When $(X,L)$ is a principally polarized 
abelian variety,
let $\widetilde{\SL(2,\Z)}$ be the central $\Z$-extension of the group
$\SL(2, \Z)$ generated by
the FMT with kernel the Poincar\'e bundle on the product $X \times X$,
$(-)\otimes L  $ and $[1]$
as a subgroup of $\Aut D^b(X)$.
So $(X \times \HX) \rtimes \widetilde{\SL(2,\Z)}$ is a subgroup
of $\Aut D^b(X)$ and
 one can canonically identify the isometric automorphisms of it with
elements from $\SL(2,\Z)$.
Any FMT $\Phi_\E $  induces a  linear isomorphism
 $\Phi_\E^{\SH} \in \GL(H^{2*}(X , \Q))$, called the cohomological Fourier-Mukai transform, and this gives rise to a
representation of the corresponding group of FMTs.
In this paper, when $X$ is
principally polarized with Picard rank one, we obtain an explicit
matrix description for this representation of $(X \times \HX) \rtimes \widetilde{\SL(2,\Z)}$ in terms of  $f_\E$. As a
result, any such induced non-trivial transform on $H^{2*}(X , \Q)$ is an
anti-diagonal matrix with respect to some suitable twisted Chern
characters. This allows us to handle the numerology in the same way as that of the classical
FMT. A matrix representation for the induced
transform of an abelian surface was also considered in \cite{YY1}.

When $\alpha \in \NS_\C(X)$ is a complexified ample class, it is
expected that $Z_{\alpha} (-) = - \int_X e^{-\alpha} \ch(-)$ defines a
central charge function of some stability condition on $X$.
The space of all stability conditions carries a natural left action
of the group $\Aut D^b(X)$. This can be defined via
a natural left action of $\Aut D^b(X)$ on $\Hom(K(X), \C)$.
When $(X,L)$ is a $g$-dimensional principally polarized abelian
variety with Picard rank one, we can view the action of $\Phi_\E \in
(X \times \HX) \rtimes \widetilde{\SL(2,\Z)}$ on $Z_\alpha$ explicitly as:
$
\Phi_\E \cdot Z_\alpha = \zeta Z_{\alpha'}
$
for some $\alpha' \in \NS_\C(X)$ and $\zeta \in \C^*$ (see
\cite[Appendix]{MYY} for the dimension 2 case). When
$\zeta$ is real one can expect that the FMT $\Phi_\E$ gives an
equivalence of some hearts of particular stability conditions of
$D^b(X)$ whose $\alpha$ and $\alpha'$ are determined by
$\Im\,\zeta=0$ (see Note \ref{prop3.2}).
For example when $g=2$, following similar ideas in \cite{Yosh1}, one
can show that any FMT gives an equivalence of two abelian categories
each of which are tilts of $\coh(X)$ (see \cite{Huy2}).
Understanding the homological FMT for the case of $g=3$ is the basis
of this paper. When the Picard rank is 1, this amounts to
understanding the transforms as a numerical matrix. This then allows
us, in a similar way to \cite{MP}, to show that any non-trivial FMT in 
$(X \times \HX) \rtimes \widetilde{\SL(2,\Z)}$
gives an equivalence of two abelian categories
each of which are double tilts of $\coh(X)$ (see Theorem \ref{prop4.5}).
Minimal objects are sent to minimal objects again under an FMT. This
enables us to obtain an inequality involving the top
part of the Chern character of minimal objects in these abelian
categories, and this is exactly the strong B-G type
inequality. Therefore any tilt stable object with zero tilt slope
satisfies the strong inequality in \cite{BMT} for our abelian
threefold case (see Theorem \ref{prop4.6}).

Toda considered similar ideas in an attempt to construct a ``Gepner''
type stability condition on a quintic threefold using the spherical
twist of the structure sheaf (see \cite{Tod3}). In \cite{Pol1},
Polishchuk tested the existence of stability
conditions for abelian varieties by studying ``Lagrangian-invariant''
objects of $D^b(X)$.

\section*{Notation}
We follow the notation of the first paper \cite{MP} which we summarize
and extend as follows.
\begin{enumerate}[(i)]
\item We will denote an $n\times n$ anti-diagonal matrix with entries $a_k$ by
$$
\adiag(a_1, \ldots, a_n)_{ij} : = \begin{cases}
a_k & \text{if } i=k, j=n+1-k \\
0 & \text{otherwise}.
\end{cases}
$$

\item For $0 \le i \le \dim X$, $\coh^{\le i}(X) : = \{E \in \coh(X): \dim \Supp(E) \le i  \}$, $\coh^{\ge i}(X) := \{E \in \coh(X): \text{for } 0 \ne F \subset E,   \ \dim \Supp(F) \ge i  \}$ and $\coh^{i}(X) : = \coh^{\le i}(X) \cap \coh^{\ge i}(X)$.

\item For an interval $I \subset \R\cup\{+\infty\}$,
$\HN^{\mu}_{\om, B}(I) := \{ E \in \coh(X) : [\mu_{\om , B}^{-}(E) , \mu_{\om , B}^{+}(E)] \subset I \}$.
Similarly the subcategory $\HN^{\nu}_{\om, B}(I) \subset \B_{\om,B}$
is defined.

\item For a Fourier-Mukai functor $\Upsilon$ and a heart $\mathfrak{A}$ of a t-structure
  for which
$D^b(X)\cong D^b(\mathfrak{A})$, $\Upsilon^k_{\mathfrak{A}}(E) :=
  H^{k}_{\mathfrak{A}}(\Upsilon(E))$.
For a sequence of integers $ i_1, \ldots, i_s$,
\[V^{\Upsilon}_{\mathfrak{A}}(i_1, \ldots, i_s) := \{ E \in D^b(X) :
  \Upsilon^j_{\mathfrak{A}}(E) = 0 \text{ for } j \notin \{i_1,
  \ldots, i_s \} \}.
\]
If $\Up$ is a Fourier-Mukai transform then $E \in \coh(X)$ being $\Upsilon$-WIT$_i$ is equivalent
  to $E \in V^{\Upsilon}_{\coh(X)}(i)$.

\item For  a $g$-dimensional polarized projective variety $(X,L)$ with Picard rank one over $\C$, the Chern character of any $E \in D^b(X)$ is of the form $(a_0, a_1 \ell, a_2 {\ell^2}/{2!}, \ldots, a_g {\ell^g}/{g!})$ for some $a_i \in
  \Z$. Here $\ell : =c_1(L)$. For simplicity we write $\ch(E) = (a_0, a_1, a_2, \ldots, a_g)$. Also
  we abuse notation to write $L^k$ for the functor $(-)\otimes
  L^k$.
\end{enumerate}
\section{Preliminaries}
\label{section1}
\subsection{Construction of stability conditions for threefolds}
Let us quickly recall the conjectural construction of stability
conditions for a given smooth projective threefold $X$ over $\C$ as
introduced in \cite{BMT}.
Let $\om, B$ be in $\NS_{\R}(X)$ with $\om$ an ample class, i.e. $B+ i
\om \in \NS_{\C}(X)$ is a complexified ample class. The twisted Chern
character with respect to $B$
 is defined by $\ch^B(-) = e^{-B} \ch (-)$.
The twisted slope $\mu_{\om , B}(E)$ of  $E \in \coh(X)$ is defined by
$$
\mu_{\om , B} (E) = \begin{cases}
+ \infty & \text{if } E \text{ is a torsion sheaf} \\
\frac{\om^{2} \ch_1^B(E)}{\ch^B_0(E)} & \text{otherwise}.
\end{cases}
$$
We say  $E \in \coh(X)$ is $\mu_{\om , B}$-(semi)stable, if for any $0 \ne
F \varsubsetneq E$, $\mu_{\om , B}(F)< (\le) \mu_{\om ,
  B}(E/F)$. The Harder-Narasimhan (H-N for short) property holds for
$\coh(X)$ and so we can define the slopes
\begin{align*}
\mu_{\om , B}^{+}(E)  = \max_{0 \ne G \subseteq E} \ \mu_{\om , B}(G), \   \   \
\mu_{\om , B}^{-}(E)  = \min_{G \subsetneq E} \ \mu_{\om , B}(E/G)
\end{align*}
of $E \in \coh(X)$. Then for a given interval $I \subset
\R\cup\{+\infty\}$, the subcategory $\HN^{\mu}_{\om, B}(I) \subset
\coh(X)$ is defined by
$$
\HN^{\mu}_{\om, B}(I) = \{ E \in \coh(X) : [\mu_{\om , B}^{-}(E) , \mu_{\om , B}^{+}(E)] \subset I \}.
$$
The subcategories  $\T_{\om , B}$ and $\F_{\om , B}$ of $\coh(X)$ are defined by
\begin{align*}
\T_{\om , B} = \HN^{\mu}_{\om, B}(0, +\infty], \ \ \
\F_{\om , B} = \HN^{\mu}_{\om, B}(-\infty, 0].
\end{align*}
Now $( \T_{\om , B} , \F_{\om , B})$ forms a torsion pair on $\coh(X)$
and let the abelian category $\B_{\om , B} = \langle \F_{\om , B}[1],
\T_{\om , B} \rangle \subset D^b(X)$ be the corresponding tilt of $\coh(X)$.
Define the central charge function $Z_{\om , B} : K(X) \to \C$ by
$Z_{\om , B}(E) = - \int_X e^{-B - i \omega} \ch(E)$.

Following \cite{BMT}, the tilt-slope $\nu_{\om , B}(E)$ of $E \in \B_{\om , B}$ is defined by
$$
\nu_{\om , B}(E) =
\begin{cases}
+\infty & \text{if } \omega^2 \ch^B_1(E) = 0 \\
\frac{\Im\,Z_{\om , B} (E)}{\omega^2 \ch^B_1(E)} & \text{otherwise}.
\end{cases}
$$
In \cite{BMT} the notion of $\nu_{\om , B}$-stability for objects in
$\B_{\om , B}$ is introduced in a similar way to $\mu_{\om ,
  B}$-stability for $\coh(X)$. Also it is proved that the abelian
category $\B_{\om , B}$ satisfies the H-N property with respect to
$\nu_{\om , B}$-stability. Then one can  define the slopes
$\nu^{+}_{\om , B}, \nu^{-}_{\om , B}$ for objects in $\B_{\om , B}$
and the subcategory $\HN^{\nu}_{\om, B}(I) \subset \B_{\om, B}$ for an
interval $I \subset \R\cup\{+\infty\}$.
The subcategories $\T_{\om , B}'$ and $\F_{\om , B}'$ of $\B_{\om, B}$ are defined by
\begin{align*}
\T_{\om , B}' = \HN^{\nu}_{\om, B}(0, +\infty], \ \ \
\F_{\om , B}' = \HN^{\nu}_{\om, B}(-\infty, 0].
\end{align*}
Then $( \T_{\om , B}' , \F_{\om , B}')$ forms a torsion pair on
$\B_{\om , B}$ and let the abelian category $\A_{\om , B} = \langle
\F_{\om , B}'[1],
\T_{\om , B}' \rangle \subset D^b(X)$ be the corresponding tilt.

\begin{conj} {\rm(\cite[Conjecture 3.2.6]{BMT})}
\label{prop1.1}
The pair $(Z_{\om , B}, \A_{\om , B} )$ is a Bridgeland stability condition on $D^b(X)$.
\end{conj}
\begin{defn}
\begin{enumerate}[(i)]
\item Let $\SC_{\om , B}$ be the class of $\nu_{\om , B}$-stable
  objects  $E \in \B_{\om , B}$ with $\nu_{\om , B}(E) = 0$.
\item Let $\SM_{\om,B}$ be the class of objects $E \in \SC_{\om,B}$
  with $\Ext^1(\OO_x, E) = 0$ for any $x \in X$.
\end{enumerate}
\end{defn}
The objects in $\SM_{\om,B}[1]$  are minimal objects in $\A_{\om, B}$ (see \cite[Lemma 2.3]{MP}).

Let us assume $B \in \NS_{\Q}(X)$ and  $\om \in \NS_{\R}(X)$ is an
ample class with $\om^2$ is rational. Then similar to the proof of
\cite[Proposition 5.2.2]{BMT} one can show that the abelian category
$\A_{\om, B}$ is Noetherian. Therefore Conjecture \ref{prop1.1} is
equivalent to the following (see \cite[Corollary 5.2.4]{BMT}).
\begin{conj}{\rm(\cite[Conjecture 3.2.7]{BMT})}
\label{prop1.2}
Any $E \in \SC_{\om , B}$ satisfies the so-called \textbf{Bogomolov-Gieseker Type Inequality}:
$$
\Re\,Z_{\om , B} (E[1]) < 0, \text{ i.e. } \ch^B_3(E) <  \frac{\om^2}{2} \ch^B_1(E).
$$
\end{conj}
Moreover in \cite{BMT} they proposed the following strong inequality for objects in $\SC_{\om,B}$.
\begin{conj}{\rm(\cite[Conjecture 1.3.1]{BMT})}
\label{prop1.3} 
Any $E \in \SC_{\om , B}$ satisfies the so-called \textbf{Strong Bogomolov-Gieseker Type Inequality}:
$$
\ch_3^B(E) \le \frac{\om^2}{18} \ch_1^B(E).
$$
\end{conj}
For any $E \in \SC_{\om,B}$, there exists $E' \in \SM_{\om,B}$  such that
$
0 \to E \to E' \to T \to 0
$
is a short exact sequence (SES for short) in $\B_{\om,B}$ for some $T \in \coh^0(X)$ (see
\cite[Proposition 2.9]{MP}). Therefore one only needs to check the
B-G (respectively, strong B-G) type inequality for objects  in $\SM_{\om,B}$.


\subsection{Autoequivalences of abelian varieties}
First of all we briefly introduce Fourier-Mukai theory 
(see \cite{BBR, Huy1} for further details). 
Let $X,Y$ be smooth projective varieties and let $p_i$, $i=1,2$ be the  projection
maps from $X \times Y$ to $X$ and $Y$, respectively.
The Fourier-Mukai functor (FM functor for short) $\Phi_{\E}^{X \to Y}:
D^b(X) \to D^b(Y)$ with kernel $\E \in D^b(X \times Y)$
is defined by
$$
\Phi_{\E}^{X \to Y}(-) = \textbf{R}p_{2*} (\E \stackrel{\textbf{L}}{\otimes} p_1^*(-)).
$$
When $\Phi_{\E}^{X \to Y}$ is an equivalence of the derived categories
it is called a Fourier-Mukai transform (FMT for short).
On the other hand Orlov's representability theorem (see \cite{Orl1})
says that any equivalence between $D^b(X)$ and $D^b(Y)$ is isomorphic
to $\Phi_{\E}^{X \to Y}$ for some $\E \in D^b(X \times Y)$.
Any FM functor $\Phi_{\E}: D^b(X) \to D^b(Y)$ induces a linear map
$\Phi^{\SH}_{\E} : H^{2*}(X, \Q)  \to   H^{2*}(Y, \Q)$ (sometimes
called the cohomological FM functor) and it is an isomorphism when $\Phi_{\E}$ is an FMT.

\begin{exmp}
\label{exmp1.1}
Let $(X, L)$ be a principally polarized abelian variety of dimension
$g$ with group operation $m:X\times X\to X$. Then the isogeny $\phi_L: X \to \HX$, $x \mapsto t_x^*L \otimes
L^{-1}$ is an isomorphism and also $\chi(L) = {\ell^g}/{g!} =1$.
In the rest of the paper let $\Phi : D^b(X) \to D^b(X)$ be the FMT
with the Poincar\'{e} line bundle $\Po = m^*L \otimes p_1^* L^{-1}
\otimes p_2^* L^{-1}$ on $X \times X$ as the kernel.
In \cite{Muk2} Mukai proved that
\begin{itemize}
\item $\Phi$ is an autoequivalence of the derived category $D^b(X)$,
\item $\Phi \circ \Phi \cong (-1)^* \id_{D^b(X)}[-g]$, and
\item $(L \circ \Phi)^3 \cong \id_{D^b(X)}[-g]$.
\end{itemize}
If we assume the Picard rank of $X$ is $1$ and write $\ch(E) = (a_0,  a_1 \ell ,  a_2
{\ell^2}/{2!}, \ldots,  a_g {\ell^g}/{g!})$,
then we have
$
\ch(\Phi(E))=\Phi^{\SH}(\ch(E)) =  (a_g,  -a_{g-1} \ell ,  a_{g-2}
   {\ell^2}/{2!}, \ldots,  (-1)^g a_0 {\ell^g}/{g!})
$ (see \cite[Lemma 9.23]{Huy1}).
\end{exmp}

Following the work of Orlov  the group $\Aut D^b(X)$ of FMTs from $X$ to $X$ can be described explicitly
as follows  (see
\cite{Orl2} and \cite[Chapter 9]{Huy1} for further
details).
Let $X,Y$ be two abelian varieties. Then one can write any morphism
$f: X \times \HX \to Y \times \widehat{Y}$ as a matrix
$
f =
\begin{pmatrix}
p & q \\
r & s
\end{pmatrix}
$
for some morphisms  $p : X \to Y$, $q : \HX \to Y$, $r : X \to
\widehat{Y}$ and $s : \HX \to \widehat{Y}$. These morphisms have duals: $\widehat
p:\widehat Y\to \HX$, $\widehat q:\widehat Y\to X$, $\widehat r:Y\to \HX$ and $\widehat
s:Y\to X$.
We associate a morphism $\widetilde{f}: Y \times \widehat{Y} \to X \times \HX$
to $f$ by setting
$
\widetilde{f} =
\begin{pmatrix}
\widehat{s} & -\widehat{q} \\
-\widehat{r} & \widehat{p}
\end{pmatrix}
$.
Then  $f$ is said to be isometric if it is an isomorphism and its
inverse $f^{-1} \cong \widetilde{f}$. When $Y=X$, we denote the group of
all isometric automorphisms of $X \times \HX$ by $U( X \times \HX)$.

Let $\Phi_{\E}^{X \to Y}$ be an FMT between two abelian varieties $X$
and $Y$ with kernel $\E \in D^b(X \times Y)$.
Let us define the map $\mu_X: X \times X \to X \times X$ by
$\mu_X(x_1,x_2) = (x_1, m(x_1, x_2))$. Let $P_X = p_{14}^*
\OO_{\Delta} \otimes p_{23}^*\Po_{X}$, where $p_{ij}$ are the
projection maps from $(X \times \HX) \times (X \times X)$,
$\OO_{\Delta}$ is the structure sheaf on the diagonal $\Delta \subset
X \times X$, and $\Po_X$ is the Poincar\'{e} bundle on $\HX \times X$.

Let $\F \in D^b(X \times Y)$ be an object such that $\Phi_{\F}^{Y \to
  X} \cong \left( \Phi_{\E}^{X \to Y} \right)^{-1}$ and let $\operatorname{Ad}_{\E}$
be the FMT from $X \times X$ to $Y \times Y$ with
kernel $\F \boxtimes \E$. Then it satisfies
$$
\Phi_{\operatorname{Ad}_{\E}(\mathcal{G})}^{Y \to Y} \cong \Phi_{\E}^{X \to Y} \circ \Phi_{\mathcal{G}}^{X \to X} \circ  \left( \Phi_{\E}^{X \to Y} \right)^{-1}
$$
for any $\mathcal{G} \in D^b(X \times X)$ (see \cite{Orl2}). Now define the
equivalence $F_{\E} : D^b(X \times \HX) \to D^b(Y \times \widehat{Y})$
by
$$
F_{\E} = \left(\Phi_{P_{Y}}^{(Y \times \widehat{Y}) \to (Y \times Y)} \right)^{-1} \circ \left(\textbf{R}\mu_{Y*}\right)^{-1} \circ \operatorname{Ad}_{\E}  \circ \ \textbf{R}\mu_{X*} \circ \Phi_{P_{X}}^{(X \times \HX) \to (X \times X)},
$$
so that $F_{\E}$ fits into the following commutative diagram (see \cite{Huy1, Orl2}).
$$
\xymatrixcolsep{5pc}
\xymatrixrowsep{2.5pc}
\xymatrix{
D^b(X \times \HX) \ar[d]_{ \Phi_{P_{X}}^{(X \times \HX) \to (X \times X)}} \ar[r]^{F_{\E}} &   D^b(Y \times \widehat{Y}) \ar[d]^{ \Phi_{P_{Y}}^{(Y \times \widehat{Y}) \to (Y \times Y)}} \\
D^b(X \times X) \ar[d]_{ \textbf{R}\mu_{X*}}  &  D^b(Y \times Y)\ar[d]^{ \textbf{R}\mu_{Y*}} \\
D^b(X \times X) \ar[r]^{\operatorname{Ad}_{\E}} & D^b(Y \times Y)
}
$$
The equivalence $F_{\E}$ can also be expressed in a simple form as follows.
\begin{lem}
\label{prop1.4} {\rm(\cite[Proposition~9.39]{Huy1}, \cite{Orl2})}
The equivalence $F_{\E}$ is isomorphic to $f_{\E*}(-) \otimes N_{\E} $
for some line bundle  $N_{\E}$ on $Y \times \widehat{Y}$ and isometric
isomorphism $f_{\E}: X \times \HX \to Y \times \widehat{Y}$.
Moreover, $f_{\E}(s , \widehat{s}) = (t, \widehat{t})$ if and only if
$\Phi_{(t, \widehat{t})} \circ \Phi_{\E}^{X \to Y} \cong \Phi_{\E}^{X \to
  Y} \circ \Phi_{(s, \widehat{s})}$. Here $\Phi_{(z, \widehat{z})} = t_{z*}(-)
\otimes \Po_{\widehat{z}}$ and $\Po_{\widehat{z}}$ is the restriction of the
Poincar\'{e} line bundle on the product $Z \times \widehat{Z}$.
\end{lem}

\begin{exmp}
\label{exmp1.2}
Let $(X, L)$ be a principally polarized abelian variety. The following
examples are important in this paper (see \cite[Examples~9.38]{Huy1}). Here
$\delta: X \to X \times X$ is the diagonal embedding.
\begin{center}
\begin{spacing}{1.5}
\begin{tabular}{ p{4.5cm} p{4.5cm} p{2cm}}
\toprule
$\Phi_\E$ & $f_\E$     & $N_\E$   \\
\midrule
$[1]= \Phi_{\OO_{\Delta}[1]}^{X \to X}$   &
      $f_{[1]}=\id_{X \times \HX}$    &
            $ \OO_{X \times \HX} $   \\
$ \Phi_{(s, \widehat{s})} = t_{s*}(-) \otimes \Po_{\widehat{s}}$ &
     $f_{(s, \widehat{s})}=\id_{X \times \HX}$    &
           $\Po_{\widehat{s}} \boxtimes \Po_{s}^{*}$ \\
 $\Phi  = \Phi_{\Po}^{X \to X}$  &
       $f_{\Po} = \begin{pmatrix}  0 & -\phi_L^{-1} \\  \phi_L & 0   \end{pmatrix} $ &
         $\Po_X$       \\
$(-)\otimes L  = \Phi_{\delta_* L}^{X \to X}$ &
    $f_{\delta_* L} = \begin{pmatrix}  1 & 0 \\ -\phi_L & 1  \end{pmatrix} $ &
        $L \boxtimes \OO_{\HX}$ \\
\bottomrule
\end{tabular}
\end{spacing}
\end{center}
\end{exmp}

Let $X,Y,Z$ be abelian varieties and let $\Phi_{\E}^{X \to Y}$,
$\Phi_{\mathcal{F}}^{Y \to Z}$, $\Phi_{\mathcal{G}}^{X \to Z}$ be FMTs
such that $\Phi_{\mathcal{G}}^{X \to Z} \cong \Phi_{\mathcal{F}}^{Y
  \to Z} \circ \Phi_{\E}^{X \to Y}$. Then one can show that
$f_{\mathcal{G}} \cong f_{\F} \circ f_{\E}$ and $N_{\mathcal{G}} \cong
N_{\F} \otimes f_{\F*} N_{\E}$. So we have a well defined group
homomorphism
$$
\sigma_X : \Aut D^b(X) \to U(X \times \HX), \ \Phi_{\E} \mapsto f_{\E}.
$$

\begin{lem}
\label{prop1.5}{\rm(\cite[Proposition~9.55]{Huy1})}
The map $\sigma_X$ is an epimorphism and its kernel consists of
autoequivalences $\Phi_{(s, \widehat{s})}[k]$ where $s \in X$, $\widehat{s}
\in \HX$ and $k \in \Z$. So $\ker \sigma_X \cong \Z \oplus (X \times
\HX)$.
\end{lem}

\begin{notation}
\rm
Assume  $(X,L)$ is a principally polarized abelian variety. 
Let $\widetilde{\SL(2,\Z)}$ be the central $\Z$-extension of the group
$\SL(2, \Z)$ generated by the FMTs $\Phi$, $(-)\otimes L$ and $[1]$
as a subgroup of $\Aut D^b(X)$.
So $(X \times \HX) \rtimes \widetilde{\SL(2,\Z)}$ is a subgroup
of $\Aut D^b(X)$. 
Consequently, we have the following diagram (see \cite[Chapter 9]{Huy1}):
$$
\xymatrixcolsep{2.25pc}
\xymatrixrowsep{2.25pc}
\xymatrix{
0 \ar[r] & \Z \oplus (X \times \HX)  \ar[r]  &  \Aut D^b(X)   \ar[r]  &  U(X \times \HX) \ar[r]  & 1 \\
0 \ar[r]  & \Z \oplus (X \times \HX) \ar[r] \ar@{=}[u]  &  (X \times \HX) \rtimes \widetilde{\SL(2,\Z)} \ar[r] \ar@{^{(}->}[u]  &   \SL(2,\Z) \ar[r] \ar@{^{(}->}[u] & 1.  \\
}
$$
The  isometric
automorphism of any FMT in $(X \times \HX) \rtimes \widetilde{\SL(2,\Z)}$ is of the form
$$
 \begin{pmatrix}
x & y \phi_L^{-1} \\
z \phi_L & w
\end{pmatrix}
$$ 
for some $x,y,z, w \in \Z$ satisfying $xw-yz =1$. In
the rest of the paper we abuse notation by dropping $\phi_L,
\phi_L^{-1}$ from this matrix. In this way we canonically identify
such isometric automorphisms  with elements of $ \SL(2,\Z)$.

\end{notation}
\section{Matrix Representations of $\GL(2,\R)$ and Induced FMTs on $ H^{2*}(X, \Q)$}
\label{section2}
\subsection{Finite dimensional matrix representations of $\GL(2,\R)$}
Following \cite{Knapp}, we explicitly construct a variant of the
symmetric power representation of all dimensions ($\ge 2$) of $\GL(2, \R)$.

For $k \ge 2$, let $V_k$ be the vector space of  homogeneous
polynomials over $\R$ in variables $u_1, u_2$ of degree $k$.
Then $V_k = \bigoplus_{r=0}^k \R \left( u_1^{k-r} u_2^r \right)$.
So the set
$$
\Omega = \left\{ u_1^k, - \binom{k}{1} u_1^{k-1} u_2, \ldots, (-1)^r \binom{k}{r} u_1^{k-r} u_2^r , \ldots, (-1)^k u_2^k \right\}
$$
is a basis of $V_k$.
Here $\binom{k}{r} =
\begin{cases}
\frac{k !}{r ! (k-r)!} & \text{ if } 0 \le r \le k,\\
0 & \text{ otherwise}.
\end{cases}
$

We have $\dim_{\R} V_k  = k+1$.
Let us define the map $\rho^{(k)} : \GL(2, \R) \to \GL(V_k)$ by
$$
\rho^{(k)}(X) \left( Q
\begin{pmatrix}
 u_1 \\
 u_2
\end{pmatrix} \right) =
 Q  \left( X^T
\begin{pmatrix}
 u_1 \\
 u_2
\end{pmatrix} \right),
$$
for $X \in \GL(2, \R)$ and $Q
\begin{pmatrix}
 u_1 \\
 u_2
\end{pmatrix} \in V_k$.
Then one can easily check that $\rho^{(k)}$ is a  $(k+1)$-dimensional
linear representation of $\GL(2, \R)$. We now explicitly compute the
matrix representation of $\rho^{(k)}$ with respect to the basis
$\Omega$. Let $X = \begin{pmatrix}
 x & y \\
 z & w
\end{pmatrix} \in \GL(2, \R)$ and let $a^{(k)}_{m,n}(x,y,z,w)$ be the
$(m,n)$-entry of $\rho^{(k)}(X)$.
By definition
\begin{multline*}
(-1)^{n-1} \binom{k}{n-1} (x u_1 + z u_2)^{k-n+1} (y u_1 + w u_2)^{n-1} \\
=  \ldots + a^{(k)}_{m,n} (-1)^{m-1}  \binom{k}{m-1} u_1^{k-m+1} u_2^{m-1} + \ldots.
\end{multline*}
By setting $\lambda = k-m-i+2$, we have the following.

\begin{prop}
\label{prop2.1}
The $(m,n)$-entry  $a_{m,n}^{(k)}(x,y,z,w)$ of $\rho^{(k)}
\begin{pmatrix}
 x & y \\
 z & w
\end{pmatrix}$ is
\begin{align*}
(-1)^{n-m} \sum_{\lambda \in \Z } \binom{k-m+1}{\lambda-1}
  \binom{m-1}{n-\lambda}  x^{k - m- \lambda + 2} y^{\lambda -1}
  z^{m-n+\lambda -1} w^{n- \lambda}.
\end{align*}
\end{prop}
Here $a^{(k)}_{m,n} (x,y,z,w)$ are polynomials of $x$, $y$, $z$, $w$
with coefficients from $\Z$. Therefore $\rho^{(k)}(\SL(2, \Z)) \subset
\GL(k+1, \Z)$.

\subsection{Induced cohomological FMTs}
We now recall some important notions from finite continued fraction
theory (see \cite{HW} for further details).
Let $\m = (m_1, \ldots, m_n)$ be a sequence of integers.
Define $s_i, t_i$ for $0 \le i \le n$  by
\begin{align*}
 &s_0=1, \ \ s_1 = m_1, \ \  s_k = m_k s_{k-1} + s_{k-2} \ (2 \le k \le n), \\
 &t_0=0, \ \ t_1 = 1,  \ \ t_k = m_k t_{k-1} + t_{k-2} \ (2 \le k \le n).
\end{align*}
The key result for us is the following standard fact which we
reproduce for the reader's convenience:
\begin{prop}
\label{prop2.2}
If we write the finite continued fraction by
\begin{equation*}
 [m_1, m_2, \ldots , m_n] =  m_1 + \cfrac{1}{\displaystyle m_2   +
   \cfrac{1}{\displaystyle \cdots + \frac{1}{\displaystyle m_n}}}
\end{equation*}
then
$ \dfrac{s_n}{s_{n-1}} = [m_n, \ldots, m_1]$, $\dfrac{t_n}{t_{n-1}} =
[m_n, \ldots, m_2]$, $\dfrac{s_n}{t_n} = [m_1, \ldots, m_n]$ and
$s_n t_{n-1} - s_{n-1} t_n = (-1)^n$.
\end{prop}

Let $(X,L)$ be a $g$-dimensional principally polarized abelian variety.
The transform $\Phi_{\m}: D^b(X) \to D^b(X)$ is defined by
\begin{equation*}
\Phi_{\m}: = \Phi \circ L^{(-1)^{n+1} m_n} \circ \Phi \circ \cdots
\circ {L}^{-m_2} \circ \Phi \circ {L}^{m_1} \circ \Phi.
\end{equation*}
Here $\Phi$ is the FMT from $X$ to $X$ with the Poincar\'{e} line
bundle $\Po$ on $X \times X$ as its kernel and $L^k$ is  $(-) \otimes
L^k$.

\begin{prop}
\label{prop2.3}
The isometric automorphism  associated to the FMT $\Phi_{\m}$ is
$$
f_{\m} =(-1)^{\frac{n(n+1)}{2}}
\begin{pmatrix}
(-1)^{n+1}t_n & (-1)^{n+1} s_n  \\
 t_{n-1}  &  s_{n-1}
\end{pmatrix}.
$$
\end{prop}
\begin{proof}
By induction on $n$.
\end{proof}

Assume the Picard rank of $X$ is one and let $\ell$ be
$c_1(L)$. As usual, we write
 $\ch(E) = (a_0,  a_1,  a_2, \ldots,  a_g)$ and so the induced transform on
 $ H^{2*}(X, \Q)$ can be expressed as a $(g+1) \times (g+1)$ invertible matrix.

\begin{exmp}
\label{exmp2.1}
The following examples of induced FMTs on  $ H^{2*}(X, \Q)$ are important in this paper.
We identify them in matrix form as images of the corresponding
isometric automorphisms under $\rho^{(g)}$ as given by Proposition
\ref{prop2.1}.
\begin{center}
\begin{spacing}{1.5}
\begin{tabular}{ p{5cm} p{6.5cm}}
\toprule
$\Phi_\E$ & $\Phi_\E^{\SH} $   \\
\midrule
$[1] = \Phi_{\OO_{\Delta}[1]}^{X \to X}$   &
     $- I_{g} =  - \rho^{(g)}\left(f_{[1]}\right)$ \\

$ \Phi_{(x, \widehat{x})} = t_{x*}(-) \otimes \Po_{\widehat{x}}$   &
     $I_{g} = \rho^{(g)}\left(f_{(x, \widehat{x})} \right)$ \\

$\Phi = \Phi_{\Po}^{X \to X}$   &
     $\adiag \left( 1,-1, \ldots, (-1)^g \right) =  \rho^{(g)}\left( f_{\Po} \right)$ \\

$(-) \otimes L = \Phi_{\delta_* L}^{X \to X}$   &
     $\binom{i-1}{j-1}_{1 \le i,j \le g+1} =  \rho^{(g)}\left( f_{\delta_* L} \right)$ \\
\bottomrule
\end{tabular}
\end{spacing}
\end{center}
\end{exmp}

Since $(X,L)$ is principally polarized, any FMT $\Phi_\E$ in $ (X \times \HX) \rtimes \widetilde{\SL(2,\Z)}$ is isomorphic to
$\Phi_{\m} \circ \Phi_{(s, \widehat{s})} \circ [p]$ for some sequence of
integers $\m$, $s \in X$, $\widehat{s} \in \HX$ and $p \in \Z$.
The  induced transform $\Phi_\E^{\SH}$ on $H^{2*}(X, \Q)$ gives a well defined group homomorphism
$$
 (X \times \HX) \rtimes \widetilde{\SL(2,\Z)} \to \GL(H^{2*}(X, \Q)), \ \ \Phi_{\E} \mapsto \Phi_{\E}^{\SH}
$$
given by $\Phi_\E^{\SH} = (-1)^p \Phi_{\m}^{\SH}$ and
$\Phi_{\m}^{\SH} =  \rho^{(g)}\left(  f_{\m}  \right)$. Also $f_{\E} = f_{\m}$.

For $\m = (m_1, \ldots, m_n)$, let $x= (-1)^{\frac{(n+1)(n+2)}{2}}
t_n$, $y= (-1)^{\frac{(n+1)(n+2)}{2}} s_n$, $z =
(-1)^{\frac{n(n+1)}{2}} t_{n-1}$ and $w=(-1)^{\frac{n(n+1)}{2}}
s_{n-1}$. By Proposition \ref{prop2.3},  the induced transform on
$H^{2*}(X, \Q)$ is
$
\Phi_{\m}^{\SH}   = \rho^{(g)}  \begin{pmatrix}
x & y \\
z &  w
\end{pmatrix}
$
and its $(m,n)$-entry is given explicitly in Proposition \ref{prop2.1}.

Up to shift, any non-trivial FMT in $ (X \times \HX) \rtimes \widetilde{\SL(2,\Z)}$ is
isomorphic to some FMT $\Phi_\E$ with a universal bundle $\E$ on $X
\times X$ as the kernel.
Therefore  $\Phi_\E^{\SH} = \rho^{(g)}  \begin{pmatrix}
x & y \\
z &  w
\end{pmatrix}$
for some $x,y, z, w \in \Z$ such that $xw-yz =1$ and $\rk(\E_{\{s\} \times X}) = \ch_0(\Phi_\E(\OO_s)) =  (-1)^gy^g>0$ for any $s \in X$.

\begin{exmp}
\label{exmp2.2}
For the case $g=2$
$$
\Phi_{\E}^{\SH}=
\begin{pmatrix}
x^2 & -2 x y  & y^2 \\
-xz & xw+yz &  -yw \\
z^2  & -2zw   & w^2
\end{pmatrix}.
$$
\end{exmp}

\begin{exmp}
\label{exmp2.3}
For the case $g=3$
$$
\Phi_{\E}^{\SH} =
\begin{pmatrix}
x^3 & -3x^2y & 3xy^2 & -y^3  \\
-x^2 z & x^2w + 2xyz  & -y^2z - 2xyw  &  y^2 w \\
 x z^2 & -yz^2 - 2xzw  & x w^2 + 2yzw & -yw^2 \\
-z^3 & 3z^2w & -3zw^2 & w^3
\end{pmatrix}.
$$
\end{exmp}

Let us simply denote the twisted Chern character $\ch^{b\ell}$ by $\ch^b$.
Since $\left( (-)\otimes L^k \right)^{\SH} = \rho^{(g)} \begin{pmatrix}
1 & 0 \\
-k &  1
\end{pmatrix}$,
we have
\begin{align*}
\ch^{-w/y} \left( \Phi_{\E}(E)\right) & = e^{w\ell/y} \ \Phi_{\E}^{\SH} \left(e^{x\ell/y }  \ch^{x/y}(E)\right) \\
& =  \rho^{(g)} \begin{pmatrix}
1 & 0 \\
-{w}/{y} &  1
\end{pmatrix}
 \begin{pmatrix}
x & y \\
z &  w
\end{pmatrix}
\begin{pmatrix}
1 & 0 \\
-{x}/{y} &  1
\end{pmatrix}
\ch^{x/y}(E).
\end{align*}

Since $xw- yz =1$, we obtain the following presentation.
\begin{thm}
\label{prop2.4}
\begin{align*}
\ch^{-w/y} \left( \Phi_{\E}(E)\right) & = \rho^{(g)} \begin{pmatrix}
0 & y \\
-{1}/{y} &  0
\end{pmatrix}
\ch^{x/y}(E) \\
& = (-1)^g y^g
\adiag \left(1, \frac{-1}{y^2}, \ldots, \frac{(-1)^{g-1}}{y^{2(g-1)}},  \frac{(-1)^g}{y^{2g}} \right)
\ch^{x/y}(E).
\end{align*}
\end{thm}

\begin{rem}
As a result of this theorem, we can see that the induced transform on
$H^{2*}(X, \Q)$ of any non-trivial FMT in $ (X \times \HX) \rtimes \widetilde{\SL(2,\Z)}$  with respect to the
appropriate twisted Chern characters looks somewhat similar to the
induced transform of $\Phi$ on $H^{2*}(X, \Q)$ with usual Chern
characters.
\end{rem}

\section{Action of FMTs on Stability Conditions}
\label{section3}
A Bridgeland stability condition $\sigma$ on a triangulated category
$\D$ consists of a stability function $Z$ together with a slicing $\PP$
of $\D$ satisfying certain axioms. Equivalently, one can define
$\sigma$ by giving a bounded t-structure on $\D$ together with a
stability function $Z$ on the corresponding heart $\A$ satisfying the
H-N property. Then $\sigma$ is usually written as the pair $(Z, \PP)$
or $(Z, \A)$. See \cite{Bri1}, \cite{Huy3} or \cite[Appendix D]{BBR} for
further details.
Let $\Upsilon  \in \Aut \D$ and let $W: K(\D) \to \C$ be a group homomorphism. Then
$$
\left( \Up \cdot W \right) ([E]) = W \left( [ \Up^{-1}(E) ] \right)
$$
defines a  left action of the  group $\Aut \D$ on $\Hom (K(\D),
\C)$. Moreover this can be extended to the natural left action of
$\Aut \D$ on the space of all stability conditions on $\D$ by defining
$\Up \cdot (Z, \A) = (\Upsilon \cdot Z , \Up(\A))$.

Let $(X, L)$ be a principally polarized $g$-dimensional abelian
variety with Picard rank one and let $\ell$ be $c_1(L)$. Then the Todd
class of $X$ is trivial and so for any object in $D^b(X)$ the Mukai vector
is  the Chern character. Any complexified class in  $\NS_{\C}(X)$ is
of the form $u \ell$ for some $u = b+ i m \in \C$,  where $b,m \in
\R$. Assume $m \ne 0$.
Define the function $Z_{u \ell}: K(X) \to \C$ by
$$
Z_{u \ell}(E) = - \int_X e^{-u \ell}  \ch(E).
$$
 If we denote the Mukai pairing on $X$ by $\langle -  , - \rangle$
then $Z_{u \ell}(E) = \langle e^{u \ell}  ,  \ch(E)   \rangle$.
It is expected  that $Z_{u \ell}$ is a central charge of some
stability condition on $X$ (see \cite[Conjecture 2.1.2]{BMT},
\cite{Pol1}).
This is already known to be true for $g=1,2$ completely, and the
authors proved the case of $m =\sqrt{3}/2$, $b= 1/2$ for $g=3$ in
\cite[Theorem 3.3]{MP}.

Let $ \Phi_{\E}$ be a non-trivial FMT in $ (X \times \HX) \rtimes \widetilde{\SL(2,\Z)}$ with kernel the universal bundle $\E$ on $X \times X$.
From the previous section, the induced transform on  $H^{2*}(X, \Q)$ is  $\Phi_{\E}^{\SH} = \rho^{(g)}
\begin{pmatrix}
 x & y \\
 z &  w
\end{pmatrix}$ for some  $x,y,z,w \in \Z$ satisfying $xw-yz =1$ and
$(-1)^g y^g >0$.  Also $e^{k \ell} = \left((-) \otimes
L^k\right)^{\SH}\left(\ch(\OO)\right)= \rho^{(g)}
\begin{pmatrix}
 1 & 0\\
-k &  1
\end{pmatrix} \ch(\OO)$.
Then
$$
\Phi_{\E}^{\SH} (e^{u \ell}) = \rho^{(g)} \begin{pmatrix}  x & y \\  z &  w \end{pmatrix}
                     \begin{pmatrix} 1 & 0 \\ -u &  1 \end{pmatrix} \ch(\OO)
                    = \rho^{(g)} \begin{pmatrix} x-yu & y \\ z-wu &  w \end{pmatrix} \ch(\OO)
$$
and from Proposition \ref{prop2.1} it is equal to $(x -yu)^g e^{(-z + wu)\ell/(x- yu)}$.

By C\u{a}ld\u{a}raru-Willertons' generalization, the cohomological FMTs are
isometries with respect to the Mukai pairing (see \cite{Cal},
\cite[Proposition 5.44]{Huy1}). Therefore for any $E \in D^b(X)$ we
have
$$
\left(\Phi_{\E} \cdot Z_{u \ell} \right)(E) = \left\langle e^{u
  \ell}  \ , \ \ch(\Phi_{\E}^{-1}(E)) \right\rangle
                                            =  \left\langle
                                            \Phi_{\E}^{\SH} (e^{u
                                              \ell})   \ , \ \ch(E)
                                            \right\rangle.
$$
So the function $Z_{u \ell} \in \Hom(K(X), \C)$ satisfies the
following key relation under the action of $\Aut D^b(X)$.

\begin{prop}
\label{prop3.1}
We have 
$
\Phi_{\E} \cdot Z_{u \ell} = \zeta \, \mathcal{Z}_{v \ell}
$
for $\zeta= (x-yu)^g $ and $v = (-z + wu)/(x- yu)$.
\end{prop}

\begin{note}
\rm
\label{prop3.2}
Assume there exist a stability condition for any complexified ample class  $\theta \ell$ with a heart $\A_{\theta \ell}$ and slicing $\PP_{\theta \ell}$ associated to the central charge function $Z_{\theta \ell}$.  
From Proposition \ref{prop3.1} for any $\phi \in \R$,  \ $\zeta \, Z_{v\ell}\left( \Phi_\E\left( \PP_{u\ell}(\phi) \right)\right) \subset \R_{>0} e^{i\pi \phi}$; that is 
$$
Z_{v\ell}\left( \Phi_\E\left( \PP_{u\ell}(\phi) \right)\right) \subset \R_{>0} e^{i\left(\pi \phi - \arg(\zeta)\right)}.
$$
So  we would expect
$$
\Phi_\E \left( \PP_{u \ell}(\phi) \right) = \PP_{v\ell}\left( \phi - \frac{\arg (\zeta)}{\pi}\right),
$$
and so 
$$
\Phi_\E \left( \PP_{u \ell}(0,\,1] \right) = \PP_{v\ell}\left( -\frac{\arg (\zeta)}{\pi},  \, -\frac{\arg (\zeta)}{\pi} +1\right].
$$
For $0\le \alpha <1$, 
$\PP_{v\ell}(\alpha, \alpha+1] =\left \langle \PP_{v\ell}(0, \alpha]\,[1] , \, \PP_{v\ell}(\alpha, 1] \right \rangle$ is a tilt of $\A_{v\ell} = \PP_{v \ell}(0,1]$
 associated to a torsion theory coming from $Z_{v \ell}$-stability. 
Therefore, one would expect $\Phi_\E(\A_{u \ell})$  is a
tilt of $\A_{v \ell}$ associated to a torsion theory coming from $Z_{v \ell}$-stability, up to shift. 

Moreover, for the nontrivial FMT
$\Phi_\E$ (i.e. $y  \ne 0$), when $\zeta = (x-yu)^g$ is real, we would expect the equivalence
(for some integer $q$)
$$
\Phi_\E(\A_{u \ell}) \cong \A_{v \ell}[q].
$$
Here $u = {x}/{y} + \lambda e^{i l\pi/g}$ and $v =- {w}/{y} -
\frac{1}{\lambda y^2 } e^{- i l\pi/g}$
for some  $l \in \Z \setminus g \Z$ and $0 \ne \lambda \in \R$.
The numerology given for $g=3$ case is very important in the rest of this paper.
\end{note}

\begin{rem}
Let us consider the numerology for $g=2$ case. For $b, m \in \Q$ with
$m>0$, recall $\B_{m \ell, b \ell}$ is the tilt of $\coh(X)$ with
respect to the torsion theory coming from twisted slope $\mu_{m \ell,
  b \ell}$-stability on $\coh(X)$. Since this category is independent
of $m$, we simply denote it  by $\B_b$.
Following the ideas in \cite{Yosh1} together with the
representation of the induced transform on $H^{2*}(X, \Q)$ with
respect to the twisted Chern characters as in Theorem \ref{prop2.4},
one can prove the equivalence
$$
\Phi_{\E}[1]\left( \B_b \right) \cong  \B_{b'},
$$
where $b = x/y$ and $b'= -w/y$ (see \cite{Huy2}).
\end{rem}

\section{Relation of FMTs to the Strong B-G Type Inequality}
\label{section4}
\subsection{Some properties of $\A_{\om, B}$}
Let $(X,L)$ be a polarized projective threefold with Picard rank one
and let $\ell$ be $c_1(L)$. Let $D, B$ be in $\NS_\Q(X)$.
Then there exists $b \in \Q$ such that $B= b \ell$. Assume $b > 0$.
Then with respect to the twisted Chern character $\ch^D$ the central charge function is
$$
Z_{\sqrt{3} B, D + B}(-) = - \int_X e^{-\left( B + i \sqrt{3}B \right)} \ch^D(-).
$$
So for $E \in D^b(X)$,   $\Im \, Z_{\sqrt{3} B, D + B}(E) = \sqrt{3} b
\ell \left( \ch_2^D(E) - b \ell \ch_1^D(E) \right)$.

\begin{prop}
\label{prop4.1}
Let $E \in \B_{\sqrt{3}B, D+B}$ and  let $E_{i} =
H^{i}_{\coh(X)}(E)$. Let $E_i^\pm$ be the  H-N semistable factors of
$E_i$ with highest and lowest $\mu_{\sqrt{3}B,D+B}$ slopes. Then we have
the following:
\begin{enumerate}[(i)]
\item if $E \in \HN^{\nu}_{\sqrt{3}B , D+B}(-\infty, 0) $ and $E_{-1}
  \ne 0$, then $\ell^2 \ch_1^D(E_{-1}^+) < 0$;
\item if $E \in \HN^{\nu}_{\sqrt{3}B , D+B}(0, +\infty]$ and $\rk(E_0)
  \ne 0$, then $\ell^2 \ch_1^D(E_0^-) > 2 b \ell^3 \ch_0^D(E_0^-)$;
  and
\item if $E$ is tilt-stable with $\nu_{\sqrt{3}B, D+B}(E) =0$, then
\begin{enumerate}
\item for $E_{-1} \ne 0$, $\ell^2 \ch_1^D(E_{-1}) \le  0$ with equality
  if and only if $\ch_2^D(E_{-1}) = 0$, and
\item for $\rk(E_0) \ne 0$, $\ell^2 \ch_1^D(E_0) \ge 2 b \ell^3
  \ch_0^D(E_0)$ with equality
  if and only if $\ch_2^D(E_0) = 2b^2 \ell^2\ch_0^D(E_0)$.
\end{enumerate}
\end{enumerate}
\end{prop}
\begin{proof}
For a slope semistable torsion free sheaf $G$,  the usual B-G
inequality in terms of the twisted Chern character is
$\left(\ch_1^D(G)\right)^2\ell \ge 2 \ch_0^D(G) \ch_2^D(G)\ell$. The
proposition follows in exactly the same way as \cite[Proposition 3.1]{MP}.
\end{proof}

We have
$
Z_{\sqrt{3} B, D - B}(-) = - \int_X e^{-\left( B + i \sqrt{3}B \right)} \ch^{D-2B}(-)
$
and  also $ \ch^{D-2B} = e^{2B} \ch^{D}$.
Therefore from the above proposition we get the following.
\begin{prop}
\label{prop4.2}
Let $E \in \B_{\sqrt{3}B, D-B}$ and  let $E_{i} =
H^{i}_{\coh(X)}(E)$. Let $E_i^\pm$ be the  H-N semistable factors of
$E_i$ with highest and lowest $\mu_{\sqrt{3}B,D-B}$ slopes. Then we have
\begin{enumerate}[(i)]
\item if $E \in \HN^{\nu}_{\sqrt{3}B , D-B}(-\infty, 0)$ and $E_{-1}
  \ne 0$, then $\ell^2 \ch_1^D(E_{-1}^+) < - 2 b \ell^3
  \ch_0^D(E_{-1}^+)$;
\item if $E \in \HN^{\nu}_{\sqrt{3}B , D-B}(0, +\infty]$ and $\rk(E_0)
  \ne 0$, then $\ell^2 \ch_1^D(E_{0}^-) > 0$; and
\item if $E$ is tilt-stable with $\nu_{\sqrt{3}B, D-B}(E) =0$, then
\begin{enumerate}
\item for $E_{-1} \ne 0$, $\ell^2 \ch_1^D(E_{-1}) \le - 2 b \ell^3 \ch_0^D(E_{-1})$ with equality
  if and only if $\ch_2^D(E_{-1}) = 2b^2 \ell^2\ch_0^D(E_{-1})$, and
\item for $\rk(E_0) \ne 0$, $\ell^2 \ch_1^D(E_0) \ge 0$ with equality
  if and only if $\ch_2^D(E_0) = 0$.
\end{enumerate}
\end{enumerate}
\end{prop}

\subsection{Relation of FMTs to stability conditions}
\label{subsec4.2}

From here onwards let $(X,L)$ be a principally polarized abelian
threefold with Picard rank one and let $\ell$ be $c_1(L)$. As before, we also abbreviate the
twisted Chern character $\ch^{b\ell}(E) = e^{-b \ell}
\ch(E)$ by $\ch^b(E)$.

Similar to \cite[Example 2.5]{MP} we can identify some examples of
minimal objects of any $\A_{\om, B}$ as follows.

\begin{exmp}
\label{exmp4.1}
Let $p, q \in \Q$ and $q>0$. There exist simple semi-homogeneous
vector bundles  $\E^{\pm}_{s}$ parameterized by $s \in X$ having the
Chern character $(u^3, u^2v , uv^2 , v^3 )$ with $u ,v \in \Z$ such
that $u>0$, $\gcd(u,v) =1$ and $v/u = {p \pm q}$.
Here the vector bundles $\E^{\pm}_{s}$ are restrictions of the
universal bundles $\E^{\pm}$ on $X \times X$ associated to FMTs.
Also $\E^{\pm}_{s}$ are slope stable (see \cite[Proposition 6.16]{Muk1}).
Then the discriminant in the sense of Dre\'{z}et
$\overline{\Delta}_{\sqrt{3}q \ell}(\E^\pm_{s}) = 0$ and so by
\cite[Proposition 7.4.1]{BMT}  $\E^+_{s} , \E^-_{s}[1] \in
\B_{\sqrt{3}q \ell, p \ell}$ are $\nu_{\sqrt{3}q \ell, p
  \ell}$-stable. Also we have $\Im \, Z_{\sqrt{3}q \ell, p
  \ell}(\E^\pm_{s}) = 0$ and $\ch_1^{p}(\E^\pm_{s}) \ne 0$ and so
$\nu_{\sqrt{3}q \ell, p \ell}(\E^+_{s}) = \nu_{\sqrt{3}q \ell, p
  \ell}(\E^-_{s}[1]) = 0$. Therefore by \cite[Lemma 2.3]{MP}
$\E^+_{s}[1] , \E^-_{s}[2] \in \A_{\sqrt{3}q \ell, p \ell}$ are
minimal objects. Moreover one can check by direct computation that
$\E^+_{s} , \E^-_{s}[1] \in \SM_{\sqrt{3}q \ell, p \ell}$ satisfy the
strong B-G type inequality.
\end{exmp}

Let $\Up$ be a non-trivial FMT in $ (X \times \HX) \rtimes \widetilde{\SL(2,\Z)}$ with kernel the  universal bundle $\E$
on $X \times X$. Then the induced transform on $H^{2*}(X, \Q)$ is
$\Up^{\SH} = \rho \begin{pmatrix}
x & y \\
z &  w
\end{pmatrix}
$ for some $x,y,z,w \in \Z$ with $xw-yz =1$ and $y <0$ (see Example
\ref{exmp2.3}). Here and in the rest of the paper we write $\rho$ for
$\rho^{(3)}$. Now we have $\ch(\E_{\{ s \} \times X}) = (-y^3, y^2 w,
-yw^2, w^3)$.  Let $\HUp$ be the FMT with kernel given by 
$\Sigma^*\E^*$, where 
$$
\Sigma: X\times X \to  X\times X: \, (x_1, x_2) \mapsto (x_2, x_1)
$$ 
switches the
factors. Then $\HUp[3]$ is the quasi-inverse of $\Up$ 
and so we have $\HUp^{\SH} = \rho \begin{pmatrix}
-w & y \\
z &  -x
\end{pmatrix}$. Also $\ch(\E^*_{X \times\{s\}}) = (-y^3, -y^2 x, -yx^2, -x^3)$.
For $g=3$ case, Theorem \ref{prop2.4} says
$$
\ch^{-w/y} \left( \Up(E)\right)  = \rho \begin{pmatrix}
0 & y \\
-{1}/{y} &  0
\end{pmatrix}
\ch^{x/y}(E)
 =  \adiag\left(- y^3, y, -\frac{1}{y}, \frac{1}{y^3} \right) \ch^{x/y}(E),
$$
and we have $\ch^{-w/y} (\E_{\{s\} \times X}) = (-y^3, 0,0,0)$ and
$\ch^{x/y} (\E_{ X \times \{s\}}^* ) = (-y^3, 0,0,0)$.

For some given $\lambda \in \Q_{>0}$, let
$$
b   = \left( \frac{x}{y} + \frac{ \lambda}{2} \right), \ m =  \frac{\sqrt{3} \lambda}{2},\
b'  = \left(- \frac{w}{y} - \frac{1}{2 \lambda y^2}\right)  \ \text{ and } \  m' =  \frac{\sqrt{3}}{2 \lambda y^2}.
$$

\begin{prop}
\label{prop4.3}
For $E \in D^b(X)$, we have
\begin{itemize}
\item $ \text{if } \ch^{x/y}(E) = (a_0, a_1,a_2, a_3) \text{ then }
  \Im \, Z_{m \ell, b \ell}(E) = \frac{3 \sqrt{3} \lambda}{2}\left(a_2
  - \lambda a_1\right), \text{ and}$
\item $\text{if } \ch^{-w/y}(E) = (a_0, a_1,a_2, a_3) \text{ then }
  \Im \, Z_{m' \ell, b' \ell}(E) = \frac{3 \sqrt{3}}{2 \lambda
    y^2}\left(a_2 + \frac{1}{\lambda y^2} a_1\right)$.
\end{itemize}
\end{prop}

The following result is a  generalization of \cite[Proposition 6.2]{MP}.
\begin{prop}
\label{prop4.4}
For $E \in D^b(X)$, we have
\begin{align*}
& \Im\,Z_{m'\ell, b' \ell}(\Up(E)) = - \frac{1}{|\lambda y|^3} \Im\,Z_{m \ell, b \ell}(E), \text{ and } \\
& \Im\,Z_{m\ell, b \ell}(\HUp[1](E)) = - |\lambda y|^3 \Im\,Z_{m' \ell, b' \ell}(E).
\end{align*}
\end{prop}
\begin{proof}
Let $\ch^{x/y}(E) = (a_0, a_1, a_2, a_3)$. Then by Proposition
\ref{prop4.3} $ \Im\,Z_{m \ell, b \ell}(E) = \frac{3 \sqrt{3}
  \lambda}{2}(a_2 - \lambda a_1)$.
By Theorem \ref{prop2.4} we have
$\ch^{-w/y}(\Up(E)) = (-y^3a_3, y a_2, -a_1/y , a_0/y^3)$.   So by Proposition \ref{prop4.3}
$$
\Im\,Z_{m'\ell, b' \ell}(\Up(E)) = \frac{3 \sqrt{3}}{2 \lambda y^2}
\left(-\frac{1}{y} a_1 + \frac{1}{\lambda y^2} ya_2 \right) =
\frac{1}{\lambda^3 y^3} \Im\,Z_{m \ell, b \ell}(E).
$$
The result follows as $y<0$. Similarly one can prove the other equality.
\end{proof}

The aim of the next sections is to prove the following equivalences of
abelian categories which generalizes \cite[Theorem 6.10]{MP}.
\begin{thm}
\label{prop4.5}
The FMTs $\Up[1]$ and $\HUp[2]$ give the equivalences of abelian categories
$$
\Up[1]\left(\A_{m \ell, b \ell}\right) \cong \A_{m' \ell, b' \ell}
\text{ and } \HUp[2]\left(\A_{m' \ell, b' \ell}\right) \cong \A_{m
  \ell, b \ell}.
$$
\end{thm}

\begin{rem}
One can see that $b,m,b',m'$ in the above theorem are exactly the
numbers given for $g=3$ case in Note \ref{prop3.2}. Moreover, the
shifts are compatible with the images of $\OO_x$ under the FMTs that
are minimal objects in the corresponding abelian categories, as
discussed in Example \ref{exmp4.1}.
\end{rem}

The notion of tilt stability can be extended from rational to real
as considered in \cite{Macri2} for
$\mathbb{P}^3$. As a result of the above theorem we get the
following.

\begin{thm}
\label{prop4.6}
The strong B-G type inequality holds for tilt stable objects of $X$ with zero tilt slope.
\end{thm}
\begin{proof}
By \cite[Proposition 2.4]{Macri2} it is enough to consider a dense
family of classes $\om = \alpha \ell$, $B= \beta \ell$ such that
$\alpha/{\sqrt{3}} \in \Q_{>0}$, $\beta \in \Q$.
Then for given $\alpha, \beta$ one can easily find $x,y \in \Z$,
$\lambda \in \Q$ such that $\gcd(x,y)=1$, $\alpha =  \sqrt{3}
\lambda/{2}$, $\beta= {x}/{y} + {\lambda}/{2}$.
Now using the Euclid algorithm and Proposition \ref{prop2.2} (for example, see  Appendix A of
\cite{BH}), one can find a non-trivial FMT $\Up$ which gives the
equivalence of abelian categories as in Theorem \ref{prop4.5}.
Therefore we only need to prove the claim for objects in $\SC_{m\ell, b\ell}$.

By \cite[Proposition 2.9]{MP} it is enough to check that the strong
B-G type inequality is satisfied by each object in $\SM_{m \ell, b
  \ell} \subset \SC_{m\ell, b\ell}$.
Moreover,  the objects  in $\{M: M \cong \E^*_{X \times\{s\}}[1]
\text{ for some } s \in X \} \subset \SM_{\om,B}$ satisfy the strong
B-G type inequality (see Example \ref{exmp4.1}).
So we only need to check the strong B-G type inequality for objects in
$\SM_{m \ell, b \ell} \setminus \{M: M \cong \E^*_{X \times\{s\}}[1]
\text{ for some } s \in X \}$.

Let  $E \in \SM_{m \ell, b \ell} \setminus \{M: M \cong \E^*_{X
  \times\{s\}}[1] \text{ for some } s \in X \}$.
Then $E[1] \in \A_{m\ell,b\ell}$ is a minimal object and so by the
equivalence in Theorem \ref{prop4.5}
$\Up[1](E[1]) \in \A_{m'\ell,b'\ell}$ is also a minimal object. So
$\Up[1](E[1]) \in \F'_{m'\ell,b'\ell}[1]$ or $\Up[1](E[1]) \in
\T'_{m'\ell,b'\ell}$.
By Proposition \ref{prop4.4}, $\Im \, Z_{m'\ell,b'\ell}(\Up[1](E[1]))
=0$. Suppose $\Up[1](E[1]) \in \T'_{m'\ell,b'\ell}$. Since $\nu_{m'\ell,b'\ell} (\Up[1](E[1])) >0$ and
$\Im \, Z_{m'\ell,b'\ell}(\Up[1](E[1]))=0$, we have $m'^2\ell^2 \ch_1^{b'}(\Up[1](E[1])) = 0$.
By  \cite[Lemma 1.1 (iii)]{MP}  $\Up[1](E[1]) \in \coh^{0}(X)$, and so $E \in \HUp[1](\coh^{0}(X))$. That is, $E$ has a filtration
of objects from $\{M: M \cong \E^*_{X \times\{s\}}[1] \text{ for some
} s \in X \}$; which is not possible.
So $\Up[1](E) \in \SC_{m'\ell,b'\ell}$. Moreover, for any $x \in X$ we have
$$
\Ext^1(\OO_s, \Up[1](E)) \cong \Hom(\OO_s, \Up[2](E)) \cong \Hom(\E^*_{X \times\{s\}}[1],E) = 0
$$
as $E \not \cong  \E^*_{X \times\{s\}}[1]$. Hence  $\Up[1](E) \in \SM_{m'\ell,b'\ell}$.

Let $\ch^{x/y}(E) = (a_0, a_1,a_2, a_3)$. Write $F=\Up[1](E)$ and let $F_i =
H^{i}_{\coh(X)}(F)$. Then $\Im \, Z_{m \ell, b
  \ell}(E) = 0$ implies $a_2 = \lambda a_1$ (see Proposition
\ref{prop4.3}).
Now the strong B-G type inequality says
$$
a_3 - \lambda^2a_1 \le 0.
$$
By Theorem \ref{prop2.4}, $\ch^{-w/y}(F) = (y^3 a_3, -y \lambda a_1,
a_1/y , - a_0/y^3)$. Then by Proposition \ref{prop4.2} we have
$$
\ell^2 \ch_1^{-w/y}(F_{-1}) \le -\frac{1}{\lambda y^2} \ell^3 \ch_0^{-w/y}(F_{-1}) \ \text{ and } \ \ell^2 \ch_1^{-w/y}(F_{0}) \ge 0.
$$
Therefore $\ell^2 \ch_1^{-w/y}(F) \ge -\frac{1}{\lambda y^2} \ell^3
\ch_0^{-w/y}(F)$. That is $- y \lambda a_1 \ge - \frac{1}{\lambda y^2}
y^3 a_3$ and so $\lambda^2 a_1 \ge a_3$ as required.
\end{proof}

Then we can deduce the main theorem of this paper:
\begin{thm}
\label{prop4.7}
Let $\alpha, \beta$ such that  $\alpha/{\sqrt{3}} \in \Q_{>0}$ and
$\beta \in \Q$. Then the pair $(\A_{\alpha \ell, \beta \ell},
Z_{\alpha\ell,\beta\ell})$ defines a Bridgeland stability condition on
$D^b(X)$.
\end{thm}

\section{Fourier-Mukai Transforms on $\coh(X)$ and $\B_{\om, B}$}
\label{section5}
Let us continue the setting introduced in subsection \ref{subsec4.2}
for the principally polarized abelian threefold $(X,L)$ with Picard
rank one.

For $E \in \coh(X)$  and $q \in \Q$, define the twisted slope
$\mu_q(E) = \mu_{{\ell}/{\sqrt{6}}, q \ell}(E)$.
If $\ch(E) = (a_0, a_1, a_2, a_3)$ then $\mu_q(E) = {a_1}/{a_0} - q$
when $a_0 \ne 0$, and $\mu_q(E) = +\infty$ when $a_0 = 0$. In the rest
of the paper we mostly use $\mu_q$ slope for coherent sheaves and we
simply write $\HN_q = \HN^{\mu}_{{\ell}/{\sqrt{6}}, q  \ell}$.
Moreover define $\T_q = \HN_q(0,+\infty]$ and $\F_q = \HN_q(-\infty, 0]$.

The isomorphisms $\HUp \circ \Up \cong \id_{D^b(X)}[-3]$ and $\Up
\circ \HUp \cong \id_{D^b(X)}[-3]$ give us the following convergence
of spectral
sequences.
\begin{mukaispecseq}
\label{mukaiss}
\begin{align*}
E_2^{p,q} & = \HUp^p_{\coh(X)}  \Up^q_{\coh(X)}(E) \Longrightarrow H^{p+q-3}_{\coh(X)}(E), \\
E_2^{p,q} & = \Up^p_{\coh(X)}  \HUp^q_{\coh(X)}(E) \Longrightarrow H^{p+q-3}_{\coh(X)}(E),
\end{align*}
for $E$. Here $\Up^i_{\coh(X)} (-) = H^i_{\coh(X)}(\Up(-))$.
\end{mukaispecseq}

For $E \in \coh(X)$, from the above spectral sequences we immediately
have $\HUp^0_{\coh(X)}  \Up^0_{\coh(X)}(E)$ $=$ $\HUp^1_{\coh(X)}
\Up^0_{\coh(X)}(E)$ $=$ $\HUp^2_{\coh(X)}  \Up^3_{\coh(X)}(E)$ $=$
$\HUp^3_{\coh(X)}  \Up^3_{\coh(X)}(E) = 0$,
$\HUp^0_{\coh(X)}  \Up^1_{\coh(X)}(E) \cong \HUp^2_{\coh(X)}
\Up^0_{\coh(X)}(E)$, and $\HUp^1_{\coh(X)}  \Up^3_{\coh(X)}(E) \cong
\HUp^3_{\coh(X)}  \Up^2_{\coh(X)}(E)$.

Let $\RR \Delta$ denote the derived dualizing functor $\RR \calHom(-,
\OO)[3]$. Let $\TUp$ be the FMT
with  kernel the universal bundle $\E^*$ on $X \times X$. As in
\cite[(3.8)]{Muk2} we have the following isomorphism.

\begin{prop}{\rm(\cite[Lemma 2.2]{PP})}
\label{prop5.1}
$$
 \RR \Delta \circ \TUp  \cong (\Up \circ \RR \Delta) [3]
$$
\end{prop}
This gives us the convergence of the following spectral sequence.
\begin{dualspecseq}
\label{dualss}
$$
\Up^p_{\coh(X)} \left( \calExt^{q+3} (E, \OO) \right) \Longrightarrow \ \ ? \ \ \Longleftarrow \calExt^{p+3}\left( \TUp^{3-q}_{\coh(X)}(E), \OO \right)
$$
for $E \in \coh(X)$.
\end{dualspecseq}
\begin{note}
\rm
Let $\ch^q(E) = (a_0,a_1, a_2, a_3)$. Then we have
$\ch^{-q}(\RR \calHom(E, \OO)) = (a_0, -a_1, a_2, -a_3)$.
Therefore for the FMT $\TUp$ we have
$\ch^{w/y}(\TUp(\OO_s)) = \ch^{w/y}(\E^*_{\{s\} \times X}) = (-y^3,0,0,0)$. So
the induced transform is $\TUp^{\SH} = \rho \begin{pmatrix}
-x & y \\
z &  -w
\end{pmatrix}$. Similar results for abelian surfaces have been
considered in \cite[Lemma 6.18]{YY1}.
\end{note}

The following proposition generalizes a series of  results in Section 4 of \cite{MP}.
\begin{prop}
\label{prop5.2}
We have the following:
\begin{enumerate}[(1)]
\item for $E \in \coh(X)$
\begin{enumerate}[(i)]
\item $\Up_{\coh(X)}^0(E)$ is a reflexive sheaf,
\item $\Up_{\coh(X)}^3(E) \in \T_{-w/y}$,
\item $\Up_{\coh(X)}^0(E) \in \F_{-w/y}$;
\end{enumerate}
\item for $E \in \T_{x/y}$
\begin{enumerate}[(i)]
\item $\Up_{\coh(X)}^3(E) = 0$,
\item if $E \in \coh^{\le 1}(X)$ then $\Up_{\coh(X)}^1(E) \in \T_{-w/y}$,
\item $\Up_{\coh(X)}^2(E) \in \T_{-w/y}$;
\end{enumerate}
\item for $E \in \F_{x/y}$
\begin{enumerate}[(i)]
\item $\Up_{\coh(X)}^0(E) = 0$,
\item $\Up_{\coh(X)}^1(E)$ is a reflexive sheaf,
\item $\Up_{\coh(X)}^1(E) \in \F_{-w/y}$.
\end{enumerate}
\end{enumerate}
\end{prop}
\begin{proof}
Proofs of (1), (2) and (3) are identical to the corresponding
propositions in \cite{MP} as listed below after replacing the Chern
characters with their twisted counterparts.
\begin{enumerate}[(1)]
\item (i) \cite[Proposition 4.5]{MP}, (ii) and (iii) \cite[Proposition 4.7]{MP}.
\item (i) \cite[Proposition 4.6]{MP}, (ii) \cite[Proposition
  4.10]{MP}, (iii) \cite[Corollary 4.17]{MP}.
\item (i) \cite[Proposition 4.6]{MP}, (ii) \cite[Proposition 4.8]{MP},
  (iii) \cite[Proposition 4.16 (i)]{MP}.
\end{enumerate}
\end{proof}

For $\lambda \in \Q$,  let $\HH_\lambda$ be the abelian subcategory of
$\coh(X)$ generated by stable semi-homogeneous bundles having the
Chern character  $(a^3, a^2b,a b^2, b^3)$ satisfying $\lambda = b/a$
and $\gcd(a,b)=1$. 
Then $\HH_0$ consists of all homogeneous bundles on $X$.

Let $H_\lambda \in \HH_\lambda$. The functor  $(-)\otimes H_{\lambda}$
is of Fourier-Mukai type with kernel $\delta_*(H_\lambda)$ on $X
\times X$, where $\delta: X \to X \times X$ is the diagonal
embedding. We abuse notation to write  $H_{\lambda}$ for the functor
$(-)\otimes H_{\lambda}$. 
If the rank of $H_\lambda$ is $r$ then the functor $H_{\lambda}$
induces a linear map on $H^{2*}(X, \Q)$ and in matrix form it is given
by 
$$
H_{\lambda }^{\SH} = r ~ \rho \begin{pmatrix}
1 & 0 \\
-\lambda  & 1
\end{pmatrix}.
$$

For some integer $n>0$, let $\Phi_j$, $j=1,\ldots, n+1$ be any
collection of FMTs
in $ (X \times \HX) \rtimes \widetilde{\SL(2,\Z)}$. For $\lambda_i \in \Q$, $i=1, \ldots, n$  let
$H_{\lambda_i} \in \HH_{\lambda_i}$. 
Consider the functor $\Pi: D^b(X) \to D^b(X)$ defined by
\begin{equation}
\label{generalFM}
\Pi  = \Phi_{n+1} \circ H_{\lambda_n} \circ \Phi_{n} \circ \cdots
\circ H_{\lambda_2} \circ \Phi_2 \circ  H_{\lambda_1} \circ
\Phi_1[p]. 
\tag{\ding{93}}
\end{equation}
Since the image of a skyscraper sheaf  $\OO_s$ under any $\Phi_i$ is a
(shift of) a semi-homogeneous bundle and being semi-homogeneous is
closed under tensoring, the image of any skyscraper sheaf $\OO_s$ under a composition of
$(-)\otimes H_{\lambda_i}$s and FMTs in $(X \times \HX) \rtimes \widetilde{\SL(2,\Z)}$ has
$\coh(X)$-cohomology concentrated in one position. So we can fix $p$
to be the unique integer for which $\Pi^i_{\coh(X)}(\OO_s) = 0$ for $i \ne
0$. 
Then $\Pi$ is an FM functor with kernel a sheaf $\U$ on $X \times X$.
Hence for any $E \in \coh(X)$, $\Pi(E)$ can have non-trivial $\coh(X)$ cohomology at $0,1,2,3$ positions only.
Also one can show that $\Pi$ induces a linear map $\Pi^{\SH}$ on $H^{2*}(X,\Q)$ given by
$$
\Pi^{\SH} = a ~ \rho \begin{pmatrix}
x & y \\
z  & w
\end{pmatrix}
$$
for some $a \in \Z_{>0}$ and $x,y,z,w \in \Q$ with $xw-yz =1$. So
$\mathcal{U}_{\{s\} \times X} = \Pi(\OO_s)$ has the Chern character 
$a(-y^3, y^2w, -yw^3, w^3)$.
Assume $\mathcal{U}_{\{s\} \times X}$ is not torsion, i.e. $y <0$.

The functor $\HPi: D^b(X) \to D^b(X)$ is defined by
$$
\HPi =  \left(\Phi_{1}\right)^{-1} \circ H_{\lambda_1}^* \circ
\left(\Phi_{2}\right)^{-1}  \circ \cdots \circ H_{\lambda_{n-1}}^*
\circ \left(\Phi_{n}\right)^{-1} \circ  H_{\lambda_n}^* \circ
\left(\Phi_{n+1}\right)^{-1} [-p-3]. 
$$
One can check that $\HPi$ is an FM functor with kernel $\Sigma^* \RR
\calHom(\U, \OO)$ on $X \times X$ (as before, $\Sigma:X \times X \to  X\times
X$ switches the factors).
Moreover, $\HPi(\OO_s) \in \coh(X)$ for any $s \in X$. So the FM
kernel of $\HPi$ is $\Sigma^* \U^*$. Also $\U$ is locally free as
$\U_{\{s\} \times X}$ and $\U_{X \times \{s\} }^*$ are locally free. 
Moreover, for any $E \in \coh(X)$, $\HPi(E)$ can have non-trivial $\coh(X)$ cohomology at $0,1,2,3$ positions only,
and $\HPi[3]$ is  left and right adjoint to $\Pi$ (and vice versa).
The FM functor $\HPi$ induces a linear map $\HPi^{\SH}$ on $H^{2*}(X, \Q)$  given by
$$
\HPi^{\SH} = a ~ \rho \begin{pmatrix}
-w & y \\
z  & -x
\end{pmatrix}.
$$

We have the isomorphisms
$$
\HPi \circ \Pi \cong H_0 [-3],  \ \ \text{and} \ \ \Pi \circ \HPi \cong H_0 [-3]
$$
for some homogeneous bundles $H_0 \in \HH_0$ with $\OO$ as a direct summand of $H_0$.
Therefore we have the convergence of spectral sequences
\begin{align*}
\label{mukaitypess}
E_2^{p,q} & = \HPi^p_{\coh(X)}  \Pi^q_{\coh(X)}(E) \Longrightarrow H^{p+q-3}_{\coh(X)}(H_0 E), \tag{\ding{61}} \\
E_2^{p,q} & = \Pi^p_{\coh(X)}  \HPi^q_{\coh(X)}(E) \Longrightarrow H^{p+q-3}_{\coh(X)}(H_0 E),
\end{align*}
for $E$. Here $\Pi^i_{\coh(X)} (-) = H^i_{\coh(X)}(\Pi(-))$.

For any $E, F \in D^b(X)$ we have
\begin{align*}
\Hom(E,F) &  \hookrightarrow \Hom(E, H_0F),  \ \text{since }  \OO \text{ is a direct summand of } H_0 \\
          &  \cong \Hom(E, \HPi \circ \Pi(F)[3])  \\
          &  \cong \Hom (\Pi(E), \Pi(F)), \ \text{from the adjointness of } \HPi[3] \text{ and } \Pi.
\end{align*}

Note that $ \RR \Delta \circ  H_{\lambda}^*  \cong H_{\lambda} \circ \RR \Delta$.
Therefore by iteratively using Proposition \ref{prop5.1} for each of the FMTs $\Phi_j$ together with the above isomorphism we have
$$
 \RR \Delta \circ \TPi  \cong (\Pi \circ \RR \Delta) [3]
$$
for some FM functor $\TPi$ which is of the form \eqref{generalFM}. Moreover, the FM kernel of $\TPi$ is $\U^*$ on $X \times X$ and the induced linear map on $H^{2*}(X , \Q)$ of $\TPi$ is
$$
\TPi^{\SH} = a ~ \rho \begin{pmatrix}
-x & y \\
z  & -w
\end{pmatrix}.
$$
The above isomorphism involving the derived dualizing functor gives us the convergence of the spectral sequence
\begin{equation}
\label{dualtypess}
\Pi^p_{\coh(X)} \left( \calExt^{q+3} (E, \OO) \right) \Longrightarrow \ \ ? \ \ \Longleftarrow \calExt^{p+3}\left( \TPi^{3-q}_{\coh(X)}(E), \OO \right),
\tag{\ding{61}\ding{61}}
\end{equation}
 for $E \in \coh(X)$.

The following proposition generalizes the results on FMTs in Proposition \ref{prop5.2} for FM functors of the form \eqref{generalFM}.
\begin{prop}
\label{prop5.3}
We have the following:
\begin{enumerate}[(1)]
\item for $E \in \coh(X)$
\begin{enumerate}[(i)]
\item $\Pi_{\coh(X)}^0(E)$ is a reflexive sheaf,
\item $\Pi_{\coh(X)}^3(E) \in \T_{-w/y}$,
\item $\Pi_{\coh(X)}^0(E) \in \F_{-w/y}$;
\end{enumerate}
\item for $E \in \T_{x/y}$
\begin{enumerate}[(i)]
\item $\Pi_{\coh(X)}^3(E) = 0$,
\item if $E \in \coh^{\le 1}(X)$ then $\Pi_{\coh(X)}^1(E) \in \T_{-w/y}$,
\item $\Pi_{\coh(X)}^2(E) \in \T_{-w/y}$;
\end{enumerate}
\item for $E \in \F_{x/y}$
\begin{enumerate}[(i)]
\item $\Pi_{\coh(X)}^0(E) = 0$,
\item $\Pi_{\coh(X)}^1(E)$ is a reflexive sheaf,
\item $\Pi_{\coh(X)}^1(E) \in \F_{-w/y}$.
\end{enumerate}
\end{enumerate}
\end{prop}
\begin{proof}
The proofs are similar to that of Proposition \ref{prop5.2} or the series of similar results in Section 4 of \cite{MP}. To illustrate the similarity, we shall give the proof for (1)(i) as follows.

Let $s \in X$. Then for $0 \le i \le 2$, we have
\begin{align*}
\Hom(\OO_s, \Pi^0_{\coh(X)}(E)[i])& \hookrightarrow \Hom(\HPi (\OO_s) , \HPi \, \Pi^0_{\coh(X)}(E)[i]) \\
                                  & \cong \Hom(\U^*_{X \times \{s\}}, \HPi^2_{\coh(X)} \Pi^0_{\coh(X)}[-2+i])
\end{align*}
from the convergence of the Spectral Sequence \eqref{mukaitypess} for $E$.
So $\Hom(\OO_s, \Pi^0_{\coh(X)}(E)) = \Ext^1(\OO_s, \Pi^0_{\coh(X)}(E))=  0$, and
\begin{align*}
\Ext^2(\OO_s, \Pi^0_{\coh(X)}(E)) & \hookrightarrow \Hom(\U^*_{X \times \{s\}}, \HPi^2_{\coh(X)} \Pi^0_{\coh(X)}(E)) \\
                   & \cong \Hom(\U^*_{X \times \{s\}}, \HPi^0_{\coh(X)} \Pi^1_{\coh(X)}(E)),  \ \ \mbox{from Spectral Sequence \eqref{mukaitypess}}\\
                   & \cong \Hom(\HPi(\OO_s), \HPi \, \Pi^1_{\coh(X)}(E)),  \\
                   & \cong \Hom(\OO_s, \Pi \,\HPi \, \Pi^1_{\coh(X)}(E)[3]), \ \ \text{from the adjointness of } \HPi[3] \text{ and } \Pi \\
                   & \cong \Hom(\OO_s,  H_0 \Pi^1_{\coh(X)}(E)).
\end{align*}
Hence $\dim \{ s \in X: \Ext^2(\OO_s, \Pi^0_{\coh(X)}(E) ) \ne 0 \} \le 0$. Therefore $\Pi^0_{\coh(X)}(E)$ is a reflexive sheaf.
\end{proof}
\begin{prop}
\label{prop5.4}
For $\lambda \in \Q_{> 0}$
\begin{enumerate}[(i)]
\item if $E \in \HN_{x/y}(0,\lambda]$ then $\Pi_{\coh(X)}^0(E) \in \HN_{-w/y}(-\infty, -\frac{1}{2\lambda y^2}]$,
\item if $E \in \HN_{x/y}[-\lambda,0]$  then $\Pi_{\coh(X)}^3(E) \in \HN_{-w/y}[\frac{1}{2\lambda y^2}, +\infty]$.
\end{enumerate}
\end{prop}
\begin{proof}
(i) \ The following proof  has a similar structure to that of \cite[Proposition 4.18]{MP}.

Let $E \in \HN_{x/y}(0,\lambda]$.
Pick a bundle $H_{-\lambda} \in \HH_{-\lambda}$ of rank $r$.
Let $\Xi$ be the FM functor defined by 
$$
\Xi = \Pi \circ H_{-\lambda} \circ \HPi[3].
$$
The induced liner map of $\Xi$ on $H^{2*}(X, \Q)$ is
$$
\Xi^{\SH} = - ra^2 ~ \rho \begin{pmatrix}
x & y \\
z & w
\end{pmatrix} \begin{pmatrix}
1 & 0 \\
\lambda & 1
\end{pmatrix} \begin{pmatrix}
-w & y \\
z & -x
\end{pmatrix} = ra^2 ~ \rho \begin{pmatrix}
1+\lambda yw & -\lambda y^2 \\
\lambda w^2 & 1-\lambda yw
\end{pmatrix}.
$$
The isomorphism $\Xi \circ \Pi \cong \Pi \circ H_{-\lambda} \circ H_0$ gives us the
convergence of spectral sequence:
$$
E_2^{p,q} = \Xi^p_{\coh(X)} \Pi^q_{\coh(X)}(E) \Longrightarrow \Pi^{p+q}_{\coh(X)} (H_{-\lambda}'E)
$$
for $E$. Here $H_{-\lambda}' = H_{-\lambda}H_0$. So $H_{-\lambda}'E \in \HN_{x/y}(-\lambda,0]$.
By (3)(i) of Proposition \ref{prop5.3} $\Pi_{\coh(X)}^0(H_{-\lambda}'E)  = 0$.
Now from the convergence of the above spectral sequence, we have
$\Xi^0_{\coh(X)} \Pi^0_{\coh(X)}(E) = 0$ and
$\Xi^1_{\coh(X)} \Pi^0_{\coh(X)}(E) \hookrightarrow \Pi_{\coh(X)}^1(H_{-\lambda}'E)$.
By (3)(iii) of Proposition~\ref{prop5.3}, $\Pi_{\coh(X)}^1(H_{-\lambda}'E) \in \HN_{-w/y}(-\infty, 0]$.
Since we have $\HN_{-w/y}(-\infty, 0]  \subset \HN_{(1-\lambda yw)/\lambda y^2}(-\infty, 0]$,

\begin{equation}
\label{firstbound}
\Xi^1_{\coh(X)} \Pi^0_{\coh(X)}(E) \in \HN_{(1-\lambda yw)/\lambda y^2}(-\infty, 0].
\tag{\ding{169}}
\end{equation}

By the H-N property $\Pi_{\coh(X)}^0(E) \in \HN_{-w/y}(-\infty, 0]$ fits into the $\coh(X)$-SES
\begin{equation}
\label{mu-ses}
0 \to F \to \Pi_{\coh(X)}^0(E) \to G \to 0,
\tag{\ding{168}}
\end{equation}
for some $F \in \HN_{-w/y}(-\frac{1}{2\lambda y^2},0]$ and $G \in
  \HN_{-w/y}(-\infty , -\frac{1}{2\lambda y^2}]$. Assume $F \ne 0$ for a
    contradiction.
Then we can write $\ch^{-w/y}(F) = (a_0, \mu a_0, a_2, a_3)$ for $0 \ge \mu > - \frac{1}{2\lambda y^2}$.

By applying the FM functor $\HPi$ to the $\coh(X)$-SES \eqref{mu-ses} we have the
following long exact sequence in $\coh(X)$:
$$
0 \to \HPi_{\coh(X)}^1(G) \to \HPi_{\coh(X)}^2(F) \to \HPi_{\coh(X)}^2\Pi_{\coh(X)}^0(E) \to \cdots.
$$
By Spectral Sequence \eqref{mukaitypess} $\HPi_{\coh(X)}^2 \Pi_{\coh(X)}^0(E) \cong \HPi_{\coh(X)}^0 \Pi_{\coh(X)}^1(E)$ 
and so by (1)(iii) of Proposition \ref{prop5.3} it is in $\HN_{x/y}(-\infty, 0]$.
Also by (3)(iii) of Proposition \ref{prop5.3} $\HPi_{\coh(X)}^1(G)
  \in \HN_{x/y}(-\infty, 0]$. Therefore $\HPi_{\coh(X)}^2(F) \in
    \HN_{x/y}(-\infty, 0]$.
By (1)(ii) of Proposition \ref{prop5.3} $\HPi_{\coh(X)}^3(F) \in \HN_{x/y}(0, +\infty]$.
Therefore $\ell^2 \ch^{x/y}_1(\HPi(F)) \le 0$, and so $ya_2 \le 0$.

Since $\Xi_{\coh(X)}^0(F) \hookrightarrow \Xi^0_{\coh(X)}
\Pi^0_{\coh(X)}(E)$ we have $\Xi_{\coh(X)}^0(F) =0$.
Moreover, since  $F \in \HN_{-w/y}(- \frac{1}{2\lambda y^2},0] =
  \HN_{-(1+\lambda yw)/\lambda y^2}(\frac{1}{2\lambda y^2},\frac{1}{\lambda y^2}]$ we have
    $\Xi_{\coh(X)}^3(F) = 0$.
Apply the FM functor $\Xi$ to  $\coh(X)$-SES  \eqref{mu-ses} and consider the long exact sequence of $\coh(X)$-cohomologies:
$$
0 \to \Xi^0_{\coh(X)}(G) \to \Xi^1_{\coh(X)}(F) \to \Xi^1_{\coh(X)} \Pi^0_{\coh(X)}(E)  \to \cdots.
$$
By \eqref{firstbound}, $\Xi^1_{\coh(X)} \Pi^0_{\coh(X)}(E) \in \HN_{(1-\lambda yw)/\lambda y^2}(-\infty, 0]$, and
by (1)(iii) of Proposition~\ref{prop5.3}, $\Xi^0_{\coh(X)}(G) \in  \HN_{(1-\lambda yw)/\lambda y^2}(-\infty, 0]$.
Therefore, $\Xi_{\coh(X)}^1(F) \in \HN_{(1-\lambda yw)/\lambda y^2}(-\infty,0]$.
By (2)(iii) of Proposition~\ref{prop5.3},
$\Xi_{\coh(X)}^2(F) \in \HN_{(1-\lambda yw)/\lambda y^2}(0,+\infty]$.
So 
$$
\ell^2 \ch_1^{(1-\lambda yw)\ell/\lambda y^2}(\Xi(F)) \ge 0.
$$   
    
On the other hand, we have
\begin{align*}
\ch^{(1-\lambda yw)/\lambda y^2}(\Xi(F)) & = ra^2 ~ \rho \begin{pmatrix}
1 & 0\\
\frac{1-\lambda yw}{\lambda y^2} & 1
\end{pmatrix} \begin{pmatrix}
1+\lambda yw & -\lambda y^2 \\
\lambda w^2 & 1-\lambda yw
\end{pmatrix} \begin{pmatrix}
1 & 0\\
\frac{w}{y} & 1
\end{pmatrix} \ch^{-w/y}(F) \\
& = ra^2 ~ \rho \begin{pmatrix}
1 & -\lambda y^2 \\
\frac{1}{\lambda y^2} & 0
\end{pmatrix} \ch^{-w/y}(F) \\
& = ra^2 \begin{pmatrix}
* & * & * & *  \\
-\frac{1}{\lambda y^2} & -2 & -\lambda y^2 & 0 \\
* & * & * & *  \\
* & * & * & *
\end{pmatrix} \begin{pmatrix}
a_0 \\
\mu a_0 \\
a_2 \\
a_3
\end{pmatrix} \\
& =  ra^2 \left(*, -2a_0 \left( \mu+ \frac{1}{2\lambda y^2}\right) - \lambda ^2y^2 a_2, *, * \right).
\end{align*}
Here $a_0 >0$, $ \left( \mu+ \frac{1}{2\lambda y^2} \right)>0$, $y a_2 \le 0$
and so $\ell^2 \ch_1^{(1-\lambda yw)/\lambda y^2}(\Xi(F))  < 0$.  This is the
required contradiction. \\

\noindent (ii) \ We shall give a proof which is similar to that of
\cite[Proposition 4.19]{MP}.

Let $E \in \HN_{x/y}[-\lambda ,0]$ for some  $\lambda  \in \Q_{> 0}$.
From Spectral Sequence \eqref{dualtypess} for $E$ we have
$$
\left( \Pi_{\coh(X)}^3(E) \right)^* \cong \TPi_{\coh(X)}^0(E^*).
$$
Here $\TPi^{\SH} = a ~ \rho  \begin{pmatrix}
-x & y \\
z & -w
\end{pmatrix}$ and we have $E^{*} \in \HN_{-x/y}[0,\lambda ]$. So by (3)(i) of Proposition \ref{prop5.3} and the above result we have
 $\TPi_{\coh(X)}^0(E^*)\in \HN_{w/y}(-\infty, -\frac{1}{2\lambda y^2}]$.
Hence $\left( \Pi_{\coh(X)}^3(E) \right)^* \in \HN_{w/y}(-\infty, -\frac{1}{2\lambda y^2}]$ and so
$\Pi_{\coh(X)}^3(E) \in \HN_{-w/y}[\frac{1}{2\lambda y^2}, +\infty]$ as required.

\end{proof}


Recall, for some fixed $\lambda \in \Q_{>0}$, $b   = \left(
\frac{x}{y} + \frac{ \lambda}{2} \right)$, $m =  \frac{\sqrt{3}
  \lambda}{2}$,
$b'  = \left(- \frac{w}{y} - \frac{1}{2 \lambda y^2}\right)$ and $m' =
\frac{\sqrt{3}}{2 \lambda y^2}$. Let $\Up, \HUp$ be the FMTs as introduced in subsection \ref{subsec4.2}.

\begin{thm}
\label{prop5.5}
We have the following:
\begin{itemize}
\item[(i)] $\Up \left( \B_{m \ell, b \ell} \right)  \subset \langle \B_{m'
  \ell, b' \ell} ,\B_{m' \ell, b' \ell}[-1], \B_{m' \ell, b' \ell}[-2]
  \rangle$, and
\item[(ii)] $\HUp [1] \left( \B_{m' \ell, b' \ell} \right)  \subset \langle
  \B_{m \ell, b \ell} ,\B_{m \ell, b \ell}[-1], \B_{m \ell, b
    \ell}[-2] \rangle$.
\end{itemize}
\end{thm}
\begin{proof}
\begin{enumerate}[(i)]
\item If $E \in \F_{b} = \HN_{x/y}(-\infty, \frac{\lambda}{2}]$ then
  by (3)(i) of Proposition \ref{prop5.2} and (i) of Proposition \ref{prop5.4}   $\Up^0_{\coh(X)}(E) \in  \F_{b'}$.
   Also by (1)(ii) of Proposition \ref{prop5.2}
$\Up^3_{\coh(X)}(E) \in \T_{-w/y} \subset \T_{b'}$. Therefore $
  \Up(E)$ has $\B_{m' \ell, b' \ell}$-cohomologies in 1,2,3
  positions. That is
$$
\Up\left( \F_{b} \right)[1] \subset \langle \B_{m' \ell, b' \ell}
,\B_{m' \ell, b' \ell}[-1], \B_{m' \ell, b' \ell}[-2] \rangle.
$$
On the other hand, if $E \in \T_b = \HN_{x/y}( \frac{\lambda}{2},
+\infty]$ then by (2)(i) of Proposition \ref{prop5.2}  $\Up^3_{\coh(X)}(E) =0$ and by (2)(iii) of Proposition \ref{prop5.2}
  $\Up^2_{\coh(X)}(E) \in \HN_{-w/y}(0, +\infty] \subset \T_{b'}$.
So $\Up(E)$ has $\B_{m' \ell, b' \ell}$-cohomologies in positions
0,1,2 only. That is
$$
\Up \left( \T_{b} \right) \subset \langle \B_{m' \ell, b' \ell} ,\B_{m' \ell, b' \ell}[-1], \B_{m' \ell, b' \ell}[-2] \rangle.
$$
Hence $\Up \left( \B_{m \ell, b \ell} \right) \subset \langle \B_{m' \ell, b' \ell} ,\B_{m' \ell, b' \ell}[-1], \B_{m' \ell, b' \ell}[-2] \rangle$,
as $\B_{m \ell, b \ell} = \langle \F_b[1] ,\T_b \rangle$.

\item  We can use (3)(i), (3)(iii), (2)(i), (1)(iii) of Proposition \ref{prop5.2}  and (ii) of Proposition \ref{prop5.4}  in a
  similar way to the above proof.
\end{enumerate}
\end{proof}
Similar to section 5 in \cite{MP} one can prove the following. In this
case, we reduce to the special case of \cite[Theorem 5.1]{MP}.
\begin{lem}
\label{prop5.6}
Let $a, b \in \Z$ be such that $a>0$ and $\gcd(a,b)=1$. Let $E$ be a
slope stable torsion free sheaf with $\ch_k^{b/a}(E) = 0$ for
$k=1,2$. Then $E^{**}$ is a slope
stable semi-homogeneous bundle with $\ch(E^{**}) = (a^3, a^2b, ab^2, b^3)$.
\end{lem}
\begin{proof}
The slope stable torsion free sheaf $E$ fits into the short exact sequence
$
0 \to E \to E^{**} \to T \to 0
$
for some $T \in \coh^{\le 1}(X)$. Now $E^{**}$ is also slope stable
and so by the usual B-G inequality $\ch_k^{b/a}(E^{**}) = 0$ for
$k=1,2$.
Now we have $\ch_k(\calEnd( E^{**}) ) = 0$ for $k=1,2$ and  $\calEnd(E^{**})$ is
a slope semistable reflexive sheaf.
By \cite[Theorem 5.1]{MP} $\calEnd(E^{**})$ is a homogeneous
bundle. Therefore $E^{**}$ is a stable semi-homogeneous bundle (see
\cite{Muk1}) and so it is a restriction of a universal bundle which is
a kernel of some FMT. Since $X$ is principally polarized its Chern
character is $(a^3, a^2b, ab^2, b^3)$ as required.
\end{proof}

\section{Equivalences of the Categories $\A_{\om,B}$ Given by FMTs}
\label{section6}
The aim of this section is to complete the proof of Theorem \ref{prop4.5}.

It will be convenient to abbreviate the FMTs $\Up$ and $\HUp[1]$ by
$\Ga$ and $\HGa$ respectively.
Then by Theorem \ref{prop5.5}, the images of an object
from $\B_{m \ell, b \ell}$ (and $\B_{m' \ell, b' \ell}$) under $\Ga$ (and $\HGa$) are  complexes whose
cohomologies with respect to $\B_{m' \ell, b' \ell}$ (and $\B_{m \ell, b \ell}$) can only be non-zero in the $0$, $1$ or $2$ positions.

The abelian category $\B_{m \ell, b \ell} = \langle \F_b[1] ,\T_b \rangle$ does not depend on $m > 0$. So in the rest of the paper we write
$$
\Ga^i_{b} (E) : = H^{i}_{\B_{m \ell, b \ell}}(\Ga(E)).
$$
We  have $\Ga \circ \HGa \cong \id_{D^b(X)}[-2]$ and $\HGa \circ \Ga \cong \id_{D^b(X)}[-2]$. This gives us the following
convergence of spectral sequences and they generalize \cite[Spectral Sequence 6.1]{MP}.
\begin{specseq}
\label{Bss}
\begin{align*}
E_2^{p,q} &= \HGa^p_{b} \Ga^q_{b'} (E) \Longrightarrow H^{p+q-2}_{\B_{m \ell, b \ell}} (E),\\
E_2^{p,q} &= \Ga^p_{b'} \HGa^q_{b} (E) \Longrightarrow H^{p+q-2}_{\B_{m' \ell, b' \ell}} (E).
\end{align*}
\end{specseq}

\begin{prop}
\label{prop6.1}
For objects $E$ we have the following:
\begin{enumerate}[(1)]
\item for $E \in \T'_{m' \ell, b' \ell}$
\begin{itemize}
\item[] (i) $H^0_{\coh(X)}(\HGa^2_{b}(E)) = 0$, and (ii) if
  $\HGa^2_{b}(E) \ne 0$ then $\Im\,Z_{m\ell,b\ell}(\HGa^2_{b}(E)) >
  0$,
\end{itemize}
\item for $E \in \F'_{m' \ell, b' \ell}$
\begin{itemize}
\item[] (i) $H^{-1}_{\coh(X)}(\HGa^0_{b}(E)) = 0$, and (ii) if
  $\HGa^0_{b}(E) \ne 0$ then $\Im\,Z_{m\ell,b\ell}(\HGa^0_{b}(E)) <
  0$,
\end{itemize}
\item for $E \in \T'_{m \ell, b \ell}$
\begin{itemize}
\item[] (i) $H^0_{\coh(X)}(\Ga^2_{b'}(E)) = 0$, and (ii) if
  $\Ga^2_{b'}(E) \ne 0$ then $\Im\,Z_{m'\ell,b'\ell}(\Ga^2_{b'}(E)) >
  0$,
\end{itemize}
\item for $E \in \F'_{m \ell, b \ell}$
\begin{itemize}
\item[] (i) $H^{-1}_{\coh(X)}(\Ga^0_{b'}(E)) = 0$, and (ii) if
  $\Ga^0_{b'}(E) \ne 0$ then $\Im\,Z_{m'\ell,b'\ell}(\Ga^0_{b'}(E)) <
  0$.
\end{itemize}
\end{enumerate}
\end{prop}

\begin{proof}
The proofs for (1), (2), (3), (4) are similar to the proofs of
\cite[Propositions 6.4, 6.5, 6.6]{MP}.
However we give proofs of some of them to illustrate the similarities.

\noindent (1)(i).  Let $E \in \T'_{m' \ell, b' \ell}$.
 For any $s \in X$,
\begin{align*}
\Hom ( \HGa^2_{b}(E) , \OO_s  ) & \cong  \Hom ( \HGa^2_{b}(E) , \HGa^2_{b}(\E_{\{s\} \times X}) ) \\
                               & \cong  \Hom ( \HGa(E) , \HGa(\E_{\{s\} \times X}) ) \\
                               & \cong  \Hom ( E, \E_{\{s\} \times X} ) = 0,
\end{align*}
since $E \in \T'_{m' \ell, b' \ell}$ and $ \E_{\{s\} \times X}  \in
\F'_{m' \ell, b' \ell}$. Therefore $H^0_{\coh(X)}(\HGa^2_{b}(E)) = 0$
as required. \medskip

\noindent (4)(ii).  Let $E \in \F'_{m \ell, b \ell}$.
From (i) of (4), we have $\Ga^0_{b'}(E) \cong A$ for some $0
  \neq A \in \T_{b'} = \HN_{-w/y}(-\frac{1}{2 \lambda y^2},+\infty]$.

Consider the convergence of the spectral sequence:
$$
E^{p,q}_2=\Ga^{p}_{\coh(X)}(H^{q}_{\coh(X)}(E)) \Longrightarrow \Ga^{p+q}_{\coh(X)}(E)
$$
for $E$.
Let $E_i = H^{i}_{\coh(X)}(E)$. Then by Proposition \ref{prop4.1},
$E_{-1} \in \HN_{x/y}(-\infty, 0]$ and so by (3)(iii) and (1)(iii) of
  Proposition \ref{prop5.2} we have
$$
\Ga^1_{\coh(X)}(E_{-1}) \in \HN_{-w/y}(-\infty,0], \text{ and } \Ga^0_{\coh(X)}(E_{0}) \in \HN_{-w/y}(-\infty,0].
$$
Therefore from the convergence of the above spectral sequence for $E$, we have
$$
A \in  \HN_{-w/y}(-\frac{1}{2 \lambda y^2},+\infty] \cap \HN_{-w/y}(-\infty,0]  = \HN_{-w/y}(-\frac{1}{2 \lambda y^2},0] .
$$
Also by (3)(ii) and (1)(i) of Proposition \ref{prop5.2}
$\Ga^1_{\coh(X)}(E_{-1})$ and  $\Ga^0_{\coh(X)}(E_{0})$ are reflexive
sheaves and so $A$ is reflexive.
Let $\ch^{-w/y}(A)= (a_0, a_1, a_2, a_3)$. Then from the usual B-G
inequalities for all the H-N semistable factors of $A$ we obtain
$a_2 + \frac{1}{2\lambda y^2} a_1 \le 0$. So we have
$$
\Im\,Z_{m'\ell, b'\ell}(\Ga^0_{b'}(E)) =  \Im\,Z_{m'\ell,
  b'\ell}(\Ga^0_{b'}(A))  = \frac{3 \sqrt{3}}{2 \lambda y^2}\left(a_2
+ \frac{1}{\lambda y^2} a_1 \right) \le 0.
$$
Equality holds when $A \in \HN_{-w/y}[0]$ with $\ch^{-w/y}(A) = (a_0, 0,
0, *)$. Then, by considering a Jordan{-}H\"{o}lder filtration
for $A$ together with Lemma \ref{prop5.6}, $A$ has a
filtration of sheaves $K_i$ each of them fits into the $\coh(X)$-SESs
$$
0 \to K_i \to \E_{\{x_i\} \times X} \to \OO_{Z_i} \to 0
$$
for some 0-subschemes $Z_i \subset X$.
Here $\Ga^0_{b'}(E) \cong A \in V^{\HGa}_{\B_{m \ell, b \ell}}(2)$
implies $A \in V^{\HUp}_{\coh(X)}(2,3)$.  An easy induction on the number
of $K_i$ in $A$  shows that
$A \in V^{\HUp}_{\coh(X)}(1,3)$ and so $A \in
V^{\HUp}_{\coh(X)}(3)$. Therefore  $Z_i =\emptyset$ for all $i$ and so
$\HGa^2_{b} \Ga^0_{b'}(E) \in \coh^0(X)$. Now consider the
convergence of the Spectral Sequence \ref{Bss} for $E$. We have
$\B_{m \ell, b \ell}$-SES
$$
0 \to \HGa^0_{b}\Ga^1_{b'}(E)  \to \HGa^2_{b}\Ga^0_{b'}(E)  \to G \to 0,
$$
where $G$ is a subobject of $E$ and so $G \in \F'_{m \ell, b \ell}$. Now
$\HGa^2_{b} \Ga^0_{b'}(E) \in \coh^0(X) \subset \T'_{m \ell, b \ell}$ implies $G =0$ and so
$\HGa^0_{b}\Ga^1_{b'}(E)  \cong \HGa^2_{b}\Ga^0_{b'}(E)$.
Then we have $\Ga^0_{b'}(E) \cong \Ga^0_{b'} \HGa^0_{b}\Ga^1_{b'}(E) =0$.
This is not possible as  $\Ga^0_{b'}(E) \ne 0$.
Therefore we have the strict inequality $\Im\,Z_{m'\ell,
  b'\ell}(\Ga^0_{b'}(E)) <0$ as required. This completes the proof.
\end{proof}

As in \cite[Lemma 6.7, Corollary 6.8, Proposition 6.9]{MP} we
obtain the following table of results for the images of $\B$-objects
under the FMTs.

\begin{center}
\begin{spacing}{1.5}
\setlength{\tabcolsep}{0.7cm}
\begin{tabular}{ c c c c}
\toprule
$E$ & $\Ga^0_{b'}(E)$ & $\Ga^1_{b'}(E)$ & $\Ga^2_{b'}(E)$  \\
\midrule
$\F'_{m \ell, b \ell}$ &  $0$ & $\F'_{m' \ell, b' \ell}$ & $\T'_{m' \ell, b' \ell}$ \\
$\T'_{m \ell, b \ell}$ &  $\F'_{m' \ell, b' \ell}$ & $\T'_{m' \ell, b' \ell}$ & $0$ \\
\midrule \midrule
$E$ & $\HGa^0_{b}(E)$ & $\HGa^1_{b}(E)$ & $\HGa^2_{b}(E)$  \\
\midrule
$\F'_{m' \ell, b' \ell}$ &  $0$ & $\F'_{m \ell, b \ell}$ & $\T'_{m \ell, b \ell}$ \\
$\T'_{m' \ell, b' \ell}$ &  $\F'_{m \ell, b \ell}$ & $\T'_{m \ell, b \ell}$ & $0$ \\
\bottomrule
\end{tabular}
\end{spacing}
\end{center}

Now we have
$\Ga[1]\left(\F'_{m \ell, b\ell}[1]\right) \subset \A_{m' \ell,
  b'\ell}$ and $\Ga[1]\left(\T'_{m \ell, b\ell}\right) \subset \A_{m'
  \ell, b'\ell}$.
Since $\A_{m \ell, b\ell} = \langle \F'_{m \ell, b\ell}[1], \T'_{m
  \ell, b\ell} \rangle$,  $\Ga[1]\left( \A_{m \ell, b\ell} \right)
\subset \A_{m' \ell, b'\ell}$.
Similarly $\HGa[1]\left(\A_{m' \ell, b'\ell}\right) \subset \A_{m \ell, b\ell}$.
The isomorphisms $\HGa[1] \circ \Ga[1] \cong  \id_{D^b(X)}$ and
$\Ga[1] \circ \HGa[1] \cong \id_{D^b(X)}$ give us the equivalences
$$
\Ga[1]\left( \A_{m \ell, b\ell} \right) \cong \A_{m' \ell, b'\ell},
\text{ and } \HGa[1]\left(\A_{m' \ell, b'\ell}\right) \cong \A_{m
  \ell, b\ell}
$$
of the abelian categories as claimed in Theorem \ref{prop4.5}.
\section*{Acknowledgements}
The authors would like to thank Arend Bayer and Tom Bridgeland for very useful discussions.
Special thanks go to the referee  for pointing out several errors  and also for giving useful comments
that led to a substantial improvement of this paper.
The second author is  grateful for the support of Principal's Career Development Scholarship programme 
and Scottish Overseas Research Student Awards Scheme of the University of Edinburgh, 
and of the World Premier International Research Center Initiative (WPI Initiative), MEXT, Japan.
This work forms a part of second author's PhD thesis completed at the University of Edinburgh.


\begin{thebibliography}{20}
\bibitem[AB]{ABL} Daniele Arcara, Aaron Bertram, \emph{Bridgeland-stable moduli spaces for K-trivial surfaces}, With an appendix by Max Lieblich, J. Eur. Math. Soc.  15 (2013), no. 1, 1\,--\,38.

\bibitem[BBMT]{BBMT}  Arend Bayer, Aaron Bertram, Emanuele Macr\`{i}, Yukinobu Toda, \emph{Bridgeland Stability conditions on threefolds II: An application to Fujita's conjecture},  	J. Algebraic Geom.  23 (2014), no. 4, 693\,--\,710.

\bibitem[BBR]{BBR}   Claudio Bartocci, Ugo Bruzzo, Daniel Hern\'{a}ndez Ruip\'{e}rez, \emph{Fourier-Mukai and Nahm transforms in geometry and mathematical physics}, Progress in Mathematics, 276. Birkh\"{a}user Boston, Boston, MA, 2009.

\bibitem[BH]{BH}  Marcello Bernardara, Georg Hein, \emph{The Euclid-Fourier-Mukai algorithm for elliptic surfaces}, Asian J. Math. 18 (2014), no. 2, 345\,--\,364.

\bibitem[BMT]{BMT}  Arend Bayer, Emanuele Macr\`{i}, Yukinobu Toda, \emph{Bridgeland stability conditions on threefolds I: Bogomolov-Gieseker type inequalities},  J. Algebraic Geom. 23 (2014), no. 1, 117\,--\,163.

\bibitem[Bri1]{Bri1}  Tom Bridgeland,  \emph{Stability conditions on triangulated categories}, Ann. of Math. (2) 166 (2007), no. 2, 317\,--\,345.

\bibitem[Bri2]{Bri2}  Tom Bridgeland,  \emph{Stability conditions on $K3$ surfaces}, Duke Math. J. 141 (2008), no. 2, 241\,--\,291.

\bibitem[CW]{Cal}  Andrei C\u{a}ld\u{a}raru, Simon Willerton,  \emph{The Mukai pairing, I: A categorical approach},  New York J. Math. 16 (2010), 61\,--\,98.

\bibitem[Clay]{Asp} Paul Aspinwall, et. al., \emph{Dirichlet branes and mirror symmetry}, Clay Mathematics Monographs, 4. American Mathematical Society, Providence, RI; Clay Mathematics Institute, Cambridge, MA, 2009.

\bibitem[HW]{HW}  Godfrey Hardy, Edward Wright, \emph{An introduction to the theory of numbers}, Sixth edition. Revised by Roger Heath-Brown and Joseph Silverman. With a foreword by Andrew Wiles. Oxford University Press, Oxford, 2008.

\bibitem[Huy1]{Huy1}  Daniel Huybrechts, \emph{Fourier-Mukai transforms in algebraic geometry}, Oxford Mathematical Monographs, The Clarendon Press, Oxford University Press, Oxford, 2006.

\bibitem[Huy2]{Huy2}  Daniel Huybrechts, \emph{Derived and abelian equivalence of $K3$ surfaces}, J. Algebraic Geom. 17 (2008), no. 2, 375\,--\,400.

\bibitem[Huy3]{Huy3}  Daniel Huybrechts, \emph{Introduction to stability conditions},  	
Moduli spaces, 179\,--\,229,  London Math. Soc. Lecture Note Ser., 411, Cambridge Univ. Press, Cambridge, 2014. 

\bibitem[Kna]{Knapp} Anthony Knapp, \emph{Representation theory of semisimple groups: An overview based on examples}. Reprint of the 1986 original. Princeton Landmarks in Mathematics. Princeton University Press, Princeton, NJ, 2001.

\bibitem[LM]{LM}  Jason Lo, Yogesh More, \emph{Some examples of tilt-stable objects on threefolds},  arXiv:1209.2749v1, 2012.

\bibitem[MP]{MP}  Antony  Maciocia, Dulip Piyaratne, \emph{Fourier-Mukai Transforms and Bridgeland Stability Conditions on Abelian Threefolds}, 	  	Algebr. Geom. 2 (2015), no. 3, 270\,--\,297.

\bibitem[MYY]{MYY} Hiroki Minamide, Shintarou Yanagida, K\={o}ta Yoshioka,
  \emph{Some moduli spaces of Bridgeland's stability conditions},
   Int. Math. Res. Not. (2014), no. 19, 5264\,--\,5327.

\bibitem[Mac1]{Macri1}  Emanuele Macr\`{i}, \emph{Stability conditions on curves},  Math. Res. Lett. 14 (2007) 657\,--\,672.

\bibitem[Mac2]{Macri2}  Emanuele Macr\`{i}, \emph{A generalized Bogomolov-Gieseker inequality for the three-dimensional projective space},  Algebra Number Theory 8 (2014), no. 1, 173\,--\,190.

\bibitem[Muk1]{Muk1}  Shigeru Mukai, \emph{Semi-homogeneous vector bundles on an Abelian variety},  J. Math. Kyoto Univ. 18 (1978), no. 2, 239\,--\,272.

\bibitem[Muk2]{Muk2}  Shigeru Mukai, \emph{Duality between $D(X)$ and $D(\widehat{X})$ with its application to Picard sheaves}, Nagoya Math. J. 81 (1981), 153\,--\,175.

\bibitem[Oka]{Okada}  So Okada, \emph{Stability manifold of $\mathbb{P}^1$}, J. Algebraic Geom.,  15  (2006),  no. 3, 487\,--\,505.

\bibitem[Orl1]{Orl1} Dmitri Orlov, \emph{Equivalences of derived categories and $K3$ surfaces}, Algebraic geometry, 7. J. Math. Sci. (New York) 84 (1997), no. 5, 1361\,--\,1381.

\bibitem[Orl2]{Orl2} Dmitri Orlov, \emph{Derived categories of coherent sheaves on abelian varieties and equivalences between them}, (Russian)  Izv. Ross. Akad. Nauk Ser. Mat.  66  (2002),  no. 3, 131\,--\,158; translation in  Izv. Math.  66  (2002),  no. 3, 569\,--\,594.

\bibitem[PP]{PP}  Giuseppe Pareschi, Mihnea Popa, \emph{ GV-sheaves, Fourier-Mukai transform, and generic vanishing},
 Amer. J. Math.  133  (2011),  no. 1, 235\,--\,271.

\bibitem[Pol]{Pol1}  Alexander Polishchuk, \emph{Phases of Lagrangian-invariant objects in the derived category of an abelian variety},  	Kyoto J. Math. 54 (2014), no. 2, 427\,--\,482.

\bibitem[Sch]{Sch}   Benjamin Schmidt, \emph{A generalized Bogomolov-Gieseker inequality for the smooth quadric threefold},  	
Bull. Lond. Math. Soc. 46 (2014), no. 5, 915\,--\,923. 

\bibitem[Tod1]{Tod1}  Yukinobu Toda, \emph{Introduction and open problems of Donaldson-Thomas theory}, Derived categories in algebraic geometry, 289\,--\,318, EMS Ser. Congr. Rep., Eur. Math. Soc., Z\"urich, 2012.

\bibitem[Tod2]{Tod2}  Yukinobu Toda, \emph{A note on Bogomolov-Gieseker type inequality for Calabi-Yau 3-folds},
   Proc. Amer. Math. Soc. 142 (2014), no. 10, 3387\,--\,3394.
   
\bibitem[Tod3]{Tod3}  Yukinobu Toda, \emph{Gepner type stability conditions on graded matrix factorizations},
   	 Algebr. Geom. 1 (2014), no. 5, 613\,--\,665.

\bibitem[YY]{YY1} Shintarou Yanagida, K\={o}ta Yoshioka, \emph{Semi-homogeneous sheaves, Fourier–Mukai transforms and moduli of stable sheaves on abelian surfaces}, 	 
J. Reine Angew. Math. 684 (2013), 31\,--\,86. 

\bibitem[Yos]{Yosh1}  K\={o}ta Yoshioka, \emph{Stability and the Fourier-Mukai transform II}, Compos. Math. 145 (2009), no. 1, 112\,--\,142.

\end{thebibliography}
\end{document}